\documentclass[final]{article}

\usepackage{amsmath}
\usepackage{amsfonts}
\usepackage{amsthm}
\usepackage[english]{babel} 
\usepackage{dsfont}
 \usepackage{amssymb}
\usepackage{graphicx}        % standard LaTeX graphics tool
 \usepackage{cite}                            % when including figure files
\usepackage{subfigure}
\usepackage{epstopdf}
\usepackage{epsfig}
\usepackage[table]{xcolor}
\usepackage{subfloat}
\usepackage{float}
\usepackage[thinlines]{easytable}
\usepackage{array}
\usepackage{mathtools}

\allowdisplaybreaks
\numberwithin{equation}{section}

\newtheorem{remark}{Remark}[section]
\newtheorem{definition}{Definition}[section]
\newtheorem{lemma}{Lemma}[section]
\newtheorem{theorem}{Theorem}[section]

\usepackage[linesnumbered,algoruled,boxed,lined]{algorithm2e}

\def\RR{\mathbb R}
\def\EE{\mathcal E}
\def\TE{\tilde{\mathcal E}}
\def\BE{\mathbb E}

\def\SS{\mathbb S}
\def\OV{\overline V}

\def\OG{\widehat\Gamma}
\def\argmin{{\rm arg}\!\min}
\def\be{\begin{equation}}
\def\ee{\end{equation}}
\def\bea{\begin{eqnarray}}
\def\eea{\end{eqnarray}}

\newtheorem{assu}{Assumption}[section]
\newtheorem{example}{Example}[section]
\newcommand{\mc}[1]{\mathcal{#1}}
\newcommand{\mb}[1]{\mathbf{#1}}
\newcommand{\la}{\langle}
\newcommand{\ra}{\rangle}
\title{Consensus-Based Optimization on Hypersurfaces: Well-Posedness and Mean-Field Limit}
\author{Massimo Fornasier \footnote{Department of Mathematics, Technical University of Munich, Boltzmannstraße 3, 85748 Garching (Munich), Germany
(massimo.fornasier{@}ma.tum.de).}\qquad
Hui Huang \footnote{Department of Mathematics, Technical University of Munich, Boltzmannstraße 3, 85748 Garching (Munich), Germany
	(hui.huang{@}tum.de).}\qquad
Lorenzo Pareschi\footnote{Department of Mathematics \& Computer Science, University of Ferrara, Via Machiavelli 30, Ferrara, 44121, Italy (lorenzo.pareschi{@}unife.it).}\qquad
Philippe S\"{u}nnen \footnote{Department of Mathematics, Technical University of Munich, Boltzmannstraße 3, 85748 Garching (Munich), Germany
	(philippe.suennen@ma.tum.de).}
}
\begin{document}
\maketitle
%\tableofcontents
\begin{abstract}
We introduce a new stochastic  differential  model for global optimization of nonconvex functions on  compact  hypersurfaces.  The model  is inspired by the stochastic Kuramoto-Vicsek system and  belongs to the class of Consensus-Based Optimization methods. In fact, particles move on the   hypersurface  driven by a drift towards an instantaneous consensus point, computed as a convex combination of the particle locations weighted by the cost function according to Laplace's principle. The consensus point represents an approximation to a global minimizer. The dynamics is further perturbed by a random vector field to favor exploration, whose variance is a function of the distance of the particles to the consensus point. In particular, as soon as the consensus is reached, then the stochastic component vanishes. In this paper, we study the well-posedness of the model and we derive rigorously its mean-field approximation for large particle limit.
%In the companion paper ``Consensus-Based Optimization on the Sphere II'' we analyze the large time behavior of the system, we prove its convergence to global minimizers in a suitable sense, and we provide several applications in machine learning.
\end{abstract}


%\newpage
\section{Introduction}

Over the last decades, large systems of interacting particles (or agents)  are widely used to investigate self-organization and collective behavior. They frequently appear in modeling phenomena such as biological swarms \cite{carrillo2010asymptotic,cucker2007emergent}, crowd  dynamics \cite{albiandco2019survey,bellomo2011modeling,bellomo2012modeling}, self-assembly of nanoparticles \cite{holm2006formation} and opinion formation \cite{albiandco2017opinion,motsch2014heterophilious,helbing2010quantitative}. Similar particle models are also used in \textit{metaheuristics} 
\cite{Aarts:1989:SAB:61990,Back:1997:HEC:548530,Blum:2003:MCO:937503.937505,Gendreau:2010:HM:1941310}, which provide empirically robust solutions to tackle hard optimization problems with fast algorithms. Metaheuristics are  methods that orchestrate an interaction between local improvement procedures and global/high level strategies, and combine random and deterministic decisions,
to create a process capable of escaping from local optima and performing a robust search of a solution space.
Starting with the groundbreaking work Ref. \cite{rastrigin63} of Rastrigin on Random Search in 1963, 
numerous mechanisms for multi-agent global optimization have been  considered, among the most prominent instances we recall the Simplex Heuristics \cite{NeldMead65}, Evolutionary Programming \cite{Fogel:2006:ECT:1202305},  the  Metropolis-Hastings sampling algorithm \cite{hastings70}, Genetic Algorithms \cite{Holland:1992:ANA:531075}, Particle Swarm Optimization (PSO) \cite{kennedy2010particle,poli2007particle}, Ant Colony Optimization (ACO) \cite{dorigo2005ant} and Simulated Annealing (SA) \cite{holley1988simulated,kirkpatrick1983optimization}. 
Despite the tremendous empirical success of these techniques, most  metaheuristical  methods  still lack  proper mathematical proof of robust convergence to global minimizers.  In fact, due to the random component of metaheuristics, their asymptotic analysis would require to discern the stochastic dependencies, which is a daunting task, especially for those methods that combine instantaneous decisions with memory mechanisms. 

Recent work  Refs. \cite{pinnau2017consensus,carrillo2018analytical} on Consensus-Based Optimization (CBO) focuses on instantaneous stochastic and deterministic decisions in order to establish a consensus among particles on the location of the global minimizers within a domain.
In view of the instantaneous nature of the dynamics, the evolution of the system can be interpreted as a system of  first order stochastic differential equations, whose large particle limit is approximated by a deterministic partial differential equation of mean-field type. The large time behavior of such a deterministic PDE can be analyzed by classical techniques of large deviation bounds and the global convergence of the mean-field model can be mathematically proven in a rigorous way for a large class of optimization problems. Certainly CBO is a significantly simpler mechanism with respect to more sophisticated metaheuristics, which can include different features including memory of past exploration. Nevertheless, it seems to be powerful and robust enough to tackle many interesting nonconvex optimizations \cite{carrillo2019consensus,fhps20-2,ha2019convergence}, which would be harder to solve by gradient descent methods that have been dominating the field of optimization, especially in relevant applications such as machine learning. In fact in many of these problems the objective function is not differentiable and CBO do not use any gradient information for their exploration. Moreover, in certain models, such as training of deep neural networks, the gradient tends to explode or vanish \cite{bengio1994learning}. Finally, although there exist situations where global optimization by gradient descent methods can be ensured under {\it ad hoc} conditions, see, e.g. Refs. \cite{Chen2019,liu2019bad,recht19}, they do not offer in general guarantees of global convergence. Instead, CBO has potential of a rather general and rigorous global asymptotic/convergence analysis.
Some theoretical gaps remain open in the analysis of CBO though, in particular the rigorous derivation of the mean-field limit, due to the difficulty in establishing bounds on the moments of the probability distribution of the particles \cite{carrillo2018analytical}.  Because of this lack of compactness, a direct proof of convergence of the stochastic particle method has been recently derived in Ref. \cite{ha2019convergence}, which does not require the mean-field limit, but at the same time  it is not  able to provide a quantitative error estimate with respect to the number of particles.

Motivated by these theoretical gaps and several potential applications in machine learning, the main task of the present work is to develop a CBO approach to solve the following constrained optimization problem
\begin{equation}
v^\ast \in \argmin\limits_{v\in \Gamma}\EE(v)\,,
\end{equation}
where $0\leq \EE:\mathbb R^{d} \to \mathbb R$ is a given continuous cost function, which we wish to minimize over a hypersurface $\Gamma$. Here we assume that $\Gamma$ is a connected smooth compact hypersurface embedded in $\RR^d$, which is represented as the $0$-level set of a signed distance function $\gamma$ with $|\gamma(v)|=\mbox{dist}(v,\Gamma)$. This means that
\begin{equation*}
\Gamma=\left\{v\in \RR^d|~\gamma(v)=0\right\}\,.
\end{equation*}
If $\partial \Gamma=\emptyset$, we also assume for simplicity that $\gamma<0$ on the interior of $\Gamma$ and $\gamma>0$ on the exterior. The gradient $\nabla\gamma$ is then the outward unit normal on $\Gamma$ wherever $\gamma$ is defined. Moreover, we assume that there exists an open neighborhood $\widehat\Gamma$ of $\Gamma$ such that $\gamma\in \mathcal{C}^3(\widehat\Gamma)$.  Such setting of hypersurface $\Gamma$ has been used in Ref. \cite{demlow2007adaptive}. 
\begin{example}
	Examples of hypersurfaces $\Gamma$ in this setting are 
	\begin{itemize}
		\item the unit sphere $\SS^{d-1}$, in which case $\gamma(v)=|v|-1$, $\nabla \gamma(v)=\frac{v}{|v|}$ and $\Delta\gamma(v)=\frac{d-1}{|v|}$;
		\item a  torus radially symmetric about the $v^d$-axis  and of inner radius $r>0$ and  external radius $R>0$ that is expressed in Cartesian coordinates as the $0$-level set of the signed distance function $\gamma(v)= \sqrt{(\sqrt{|v|^2-(v^d)^2}-R)^2+(v^d)^2}-r$, where $v=(v^1,\dots,v^d)$.
	\end{itemize}
\end{example}

In particular, we consider a system of $N$ interacting particles  $((V_t^i)_{t\geq 0 })_{i=1,\dots,N}$ satisfying  the following stochastic  differential  dynamics expressed in It\^{o}'s form
\begin{align} \label{stochastic Kuramoto-Vicsek}
dV_t^i &= -\lambda P(V_t^i)(V_t^i-v_{\alpha,\EE}(\rho_t^N))dt + \sigma |V_t^i - v_{\alpha,\EE}(\rho_t^N)| P(V_t^i)dB_t^i-\frac{\sigma^2}{2}(V_t^i-v_{\alpha,\EE}(\rho_t^N))^2\Delta\gamma(V_t^i)\nabla\gamma(V_t^i)dt\,,
\end{align}
where $\lambda>0$ is a suitable drift parameter, $\sigma>0$ a diffusion parameter,
\begin{equation}
\rho_t^N=\frac{1}{N}\sum_{i=1}^{N}\delta_{V_t^i}\,,
\end{equation}
is the empirical measure of the particles ($\delta_v$ is the Dirac measure at $v\in\RR^d$), and
\begin{equation}\label{ValphaE}
v_{\alpha,\EE}(\rho_t^N)= \frac{\sum_{j=1}^{N} V^j_t e^{-\alpha \EE( V^j_t)}}{\sum_{j=1}^{N}e^{-\alpha \EE( V^j_t )}}=\frac{\int_{\mathbb R^{d}}v\omega_\alpha^\EE(v)d\rho_t^N}{\int_{\mathbb R^{d}}\omega_\alpha^\EE(v)d\rho_t^N} \quad \mbox{ with }\quad  \omega_\alpha^\EE(v):=e^{-\alpha\EE(v)}\,.
\end{equation}
For simplicity,  we write $v^2$ to mean $|v|^2$ for any vector $v \in \mathbb R^d$,  and $|v|$ is the standard Euclidean norm.
This stochastic system is considered complemented with independent and identical distributed (i.i.d.) initial data $V_0^i\in\Gamma$ with $i=1,\cdots,N$, and the common law is denoted by $\rho_0\in \mc{P}(\Gamma)$. The trajectories $((B_t^i)_{t\geq0})_{i=1,\dots N}$ denote $N$ independent standard Brownian motions in $\RR^d$.
In \eqref{stochastic Kuramoto-Vicsek} the projection operator $P(\cdot)$ is defined by
\begin{equation}\label{defP}
P(v)=I-\nabla\gamma(v)\nabla\gamma(v)^T\,.
\end{equation}
In the case of sphere we know $\nabla\gamma(v)=\frac{v}{|v|}$, so one has $P(v)=I-\frac{vv^T}{|v|^2}$,  and \eqref{stochastic Kuramoto-Vicsek} corresponds to a stochastic Kuramoto-Vicsek type model \cite{kuramoto1975self,Vicsek1995NovelTO,bolley2012mean}.

% which guarantees  all particles staying on $\Gamma$.  It is easy to check that
% \begin{equation}
% P(v)v=0
% \end{equation}
% and
% \begin{equation}\label{P2}
% v\cdot P(v)y=0
% \end{equation}
% hold for any $y\in \RR^d$.
The choice of the weight function $\omega_\alpha^\EE$ in \eqref{ValphaE} comes from  the  well-known Laplace's principle \cite{miller2006applied,Dembo2010,pinnau2017consensus}, a classical asymptotic method for integrals, which states that for any probability measure $\rho\in\mc{P}( \Gamma )$, it holds\footnote{
	For the sake of the reader we provide a simple argument for the validity of this well-known formula: let us take the exponential on the right-hand-side limit to obtain $\lim\limits_{\alpha\to\infty} \frac{1}{\left(\int_{ \Gamma }e^{-\alpha\EE(v)}d\rho(v)\right)^{1/\alpha}} =({\rm ess}\sup_{v \in  {\rm supp }(\rho)} e^{-\EE(v)})^{-1} = {\rm ess}\inf_{v \in  {\rm supp }(\rho)} e^{\EE(v)} $. The first equality holds by the well-known limit of $L^\alpha$-norms to the $L^\infty$-norm for which $\lim_{\alpha\to \infty} \| e^{-\EE(\cdot)}\|_\alpha =  \| e^{-\EE(\cdot)}\|_\infty$. By applying now the logarithm and using its continuity, one obtains \eqref{lap_princ}.}
\begin{equation}\label{lap_princ}
\lim\limits_{\alpha\to\infty}\left(-\frac{1}{\alpha}\log\left(\int_{ \Gamma }e^{-\alpha\EE(v)}d\rho(v)\right)\right)=\inf\limits_{v \in \rm{supp }(\rho)} \EE(v)\,.
\end{equation}

Let us discuss the mechanism of the dynamics. The right-hand-side of the equation \eqref{stochastic Kuramoto-Vicsek} is made of three terms. The first deterministic term $-\lambda P(V_t^i)(V_t^i-v_{\alpha,\EE}(\rho_t^N))dt $ imposes a drift to the dynamics towards $v_{\alpha,\EE}$, which is the current consensus point at time $t$ as approximation to the global minimizer. 
This drift term is the projection onto the tangent space to the hypersurfaces at $V_t^i$ of the vector $v_{\alpha,\mathcal{E}}-V_t^i$. Hence $V_t^i$ gets drifted in the direction of $v_{\alpha,\mathcal{E}}$, which is not normal to the hypersurface, and the term disappears when $V_t^i=v_{\alpha,\mathcal{E}}$. In fact, the consensus point $v_{\alpha,\EE}$ is explicitly computed as in \eqref{ValphaE} and it may lay in general outside $\Gamma$. This choice of an embedded weighted barycenter is  very simple, compatible with fast computations, and, for compact manifolds $\Gamma$, it is a good proxy for a minimizer $v^*$. One could alternatively consider the computation of a weighted barycenter on the manifold
$$
v_{\alpha,\mathcal E}^\Gamma(\rho_t) = \arg \min_{v \in \Gamma} \int_\Gamma d_\Gamma(v,w)^p e^{-\alpha \mathcal E(w)} d\rho_t(w),
$$
where $d_\Gamma$ is a (Riemannian) distance on $\Gamma$, $p\geq 1$, and $\rho_t$ is the particle distribution. However, the computation of $v_{\alpha,\mathcal E}^\Gamma$ is in general not explicit (because the distance function might not be simple or explicit) and one may have to solve at each time $t$ an optimization problem over the manifold in order to compute $v_{\alpha,\mathcal E}^\Gamma$. In view of the potentially variable curvature of $\Gamma$, the distance $v \to d_\Gamma(v,w)$ could locally behave in general as a geodesically convex or nonconvex function, rendering such a global optimization rather difficult in general. These are reasons for choosing the embedded alternative \eqref{ValphaE}.
The second stochastic term $\sigma |V_t^i - v_{\alpha,\EE}(\rho_t^N)| P(V_t^i)dB_t^i$ introduces a random decision to favor exploration, whose variance is a function of the distance of the particles to the current consensus points. In particular, as soon as consensus is reached, then the stochastic component vanishes. The last term  $-\frac{\sigma^2}{2}(V_t^i-v_{\alpha,\EE}(\rho_t^N))^2\Delta\gamma(V_t^i)\nabla\gamma(V_t^i)dt$  combined with $P(\cdot)$ is needed to ensure that the dynamics stays on the hypersurface despite the Brownian motion component.   Intuitively, it simply compensates the normal direction of the noise produced by the isotropic Brownian motion and leaves the tangential component active.  In fact, we will initially analyze the well-posedness of the system \eqref{stochastic Kuramoto-Vicsek} as defined in the whole space $\RR^d$ instead of being immediately considered constrained on the hypersurface $\Gamma$, and we will check afterwards that, if the initial data $(V_0^i)_{i=1,\dots,N}$ is set on the hypersurface, then the particles $((V_t^i)_{t\geq 0})_{i=1,\dots,N}$ remain there at all times. 
We further notice that the dynamics does not make use of any derivative of $\EE$, but only of its pointwise evaluations, which appear integrated in \eqref{ValphaE}. 
Hence, the equation can be in principle numerically implemented at discrete times also for cost functions $\EE$ which are just continuous and with no further smoothness. We require below more regularity to $\EE$ exclusively to ensure formal well-posedness of the evolution.

The main results of this paper are about the  well-posedness of \eqref{stochastic Kuramoto-Vicsek} and its rigorous mean-field limit - which is an open issue for unconstrained CBO \cite{carrillo2018analytical} -  to the following nonlocal, nonlinear Fokker-Planck equation
\begin{equation}\label{PDE}
\partial_t \rho_t= \lambda \nabla_{\Gamma} \cdot (P(v)(v-v_{\alpha,\EE}(\rho_t)) \rho_t)+\frac{\sigma^2}{2}\Delta_{\Gamma} (|v-v_{\alpha,\EE}(\rho_t) |^2\rho_t),\quad t>0,~v\in\Gamma\,,
\end{equation}
with the initial data $\rho_0\in\mc{P}(\Gamma)$. Here $\rho=\rho(t,v)\in \mc{P}(\Gamma)$ is a Borel probability measure on $\Gamma$ and
\[
v_{\alpha,\EE}(\rho_t)  = \frac{\int_{\Gamma} v \omega_\alpha^\EE(v)\,d\rho_t}{\int_{\Gamma} \omega_\alpha^\EE(v)\,d\rho_t}.
\]
The operators $\nabla_{\Gamma} \cdot$ and $\Delta_{\Gamma} $ denote the divergence and Laplace-Beltrami operator on the hypersurface $\Gamma$ respectively. The mean-field limit will be achieved through the coupling method \cite{sznitman1991topics,fetecau2019propagation,huang2017error,huang2019learning,bolley2011stochastic} by introducing  an auxiliary mono-particle process, satisfying the self-consistent nonlinear SDE  
\begin{align} \label{selfprocess}
d\overline V_t=-\lambda  P(\overline V_t)(\OV_t-v_{\alpha,\EE} (\rho_t)) dt + \sigma |\overline V_t - v_{\alpha,\EE}(\rho_t)  | P(\overline V_t)dB_t-\frac{\sigma^2}{2}(\overline V_t-v_{\alpha,\EE} (\rho_t) )^2\Delta \gamma(\OV_t)\nabla\gamma(\OV_t)dt\,,
\end{align}
with the initial data $\OV_0$ distributed according to $\rho_0\in\mc{P}(\Gamma)$. Here we  require $\rho_t=\rm{law}(\overline V_t)$, and we will show that $\rho_t$ as a measure concentrated on $\Gamma$ solves the PDE \eqref{PDE}.
We call the SDE \eqref{selfprocess} the mean-field dynamic, and the PDE \eqref{PDE} is called the  mean-field PDE. In our follow up paper Ref. \cite{fhps20-2}, for more regular datum $\rho_0$, we also prove existence and uniqueness of distributional solutions $\rho \in L^2([0,T],H^1(\Gamma))$ at any finite time $T>0$.
This model and the analysis we provide in this paper are very much inspired by the homogeneous version of the kinetic Kolmogorov-Kuramoto-Vicsek model  on the sphere $\Gamma=\mathbb S^{d-1}$, which was formally derived by Degond and Motsch \cite{demo07}  as a mean-field limit of the discrete Vicsek model \cite{alhu03,coetal02,Vicsek1995NovelTO}. Recently, the stochastic Vicsek model has received extensive attention in the mathematical community to establish its mean-field limit, hydrodynamic limit, and phase transitions. Bolley, Ca\~{n}izo and Carrillo \cite{bolley2012mean} have rigorously justified the mean-field limit in case of smoothed force field, and Gamba and Kang \cite{gaka15} and Figalli, Kang, and Morales \cite{fikamo18} extended the global existence and uniqueness of weak solutions of the kinetic Kolmogorov-Kuramoto-Vicsek  model to the case of singular mean-field force field. Recently, Degond, Frouvelle and Liu \cite{defrli13} provided also a complete and rigorous description of phase transitions such as the number and nature of equilibria, stability, convergence rate, phase diagram and hysteresis. Our results build very much upon these developments.

%{\color{red} 
Before starting our analysis of \eqref{stochastic Kuramoto-Vicsek}, \eqref{selfprocess}, and \eqref{PDE}, let us illustrate
numerically the behavior of the dynamics in the case of the minimum solution of the Ackley function for $d=3$  constrained over an hypersurface $\Gamma$ 
\be
\EE(v)= -A \exp\left(-a\sqrt{\frac{b^2}{d}\sum_{k=1}^{d} (v_k-v^*_k)^2}\right)-\exp\left(\frac1{d}\sum_{k=1}^{d} \cos(2\pi b(v^k-v^k_*))\right)+e+B,
\ee
with $A=20$, $a=0.2$, $b=3$, $B=20$ and $v=(v^1,\ldots,v^d)^T$  with $\gamma(v)=0$. 

The global minimum is at $v=v_*$.  In Figure \ref{fg:ackley} we report in the case $d=3$ the Ackley function constrained over the half sphere $v^3 \geq 0$ and over the torus with external radius $R=1$ and inner radius $r=0.5$. The colors over the surface represent the values attained by the function.
%Note that, this problem differs from the standard minimization of the Ackley function over the whole space $\RR^d$ since the sKV-CBO model operates through unitary vectors over the hyperhypersurface.

%over the angular coordinates
%\be
%\EE(v)= -A \exp\left(-a\sqrt{\frac{b^2}{d-1}\sum_{k=1}^{d-1} (\theta_k-\phi_k)^2}\right)-\exp\left(\frac1{d-1}\sum_{k=1}^{d-1} \cos(2\pi b(\theta_k-\phi_k))\right)+e+B,
%\ee
%with $A=20$, $a=0.2$, $b=6/\pi$, $B=20$ and $v=(v_1,\ldots,v_d)^T$ with
%\begin{align}
%v_1 = \cos(\theta_1),\qquad
%v_k = \cos(\theta_k) \prod_{m=1}^{k-1}\sin(\theta_m),\quad k=2,\ldots,d-1,\qquad
%v_d = \prod_{m=1}^{d-1}\sin(\theta_m),
%\label{eq:sphtra}
%\end{align}
%where $\theta_1,\ldots,\theta_{d-2}\in [0,\pi)$ and $\theta_{d-1}\in [0,2\pi)$.
%The global minimum is attained at $\theta^*_k=\phi_k$, $\phi_k\in [0,\pi)$, $k=1,\ldots,d-1$. In Figure \ref{fg:ackley} we report the Ackley function for $d=3$ and $\theta_{1},\theta_2 \in [0,\pi)$.}

\begin{figure}[tb]
	\includegraphics[scale=0.39]{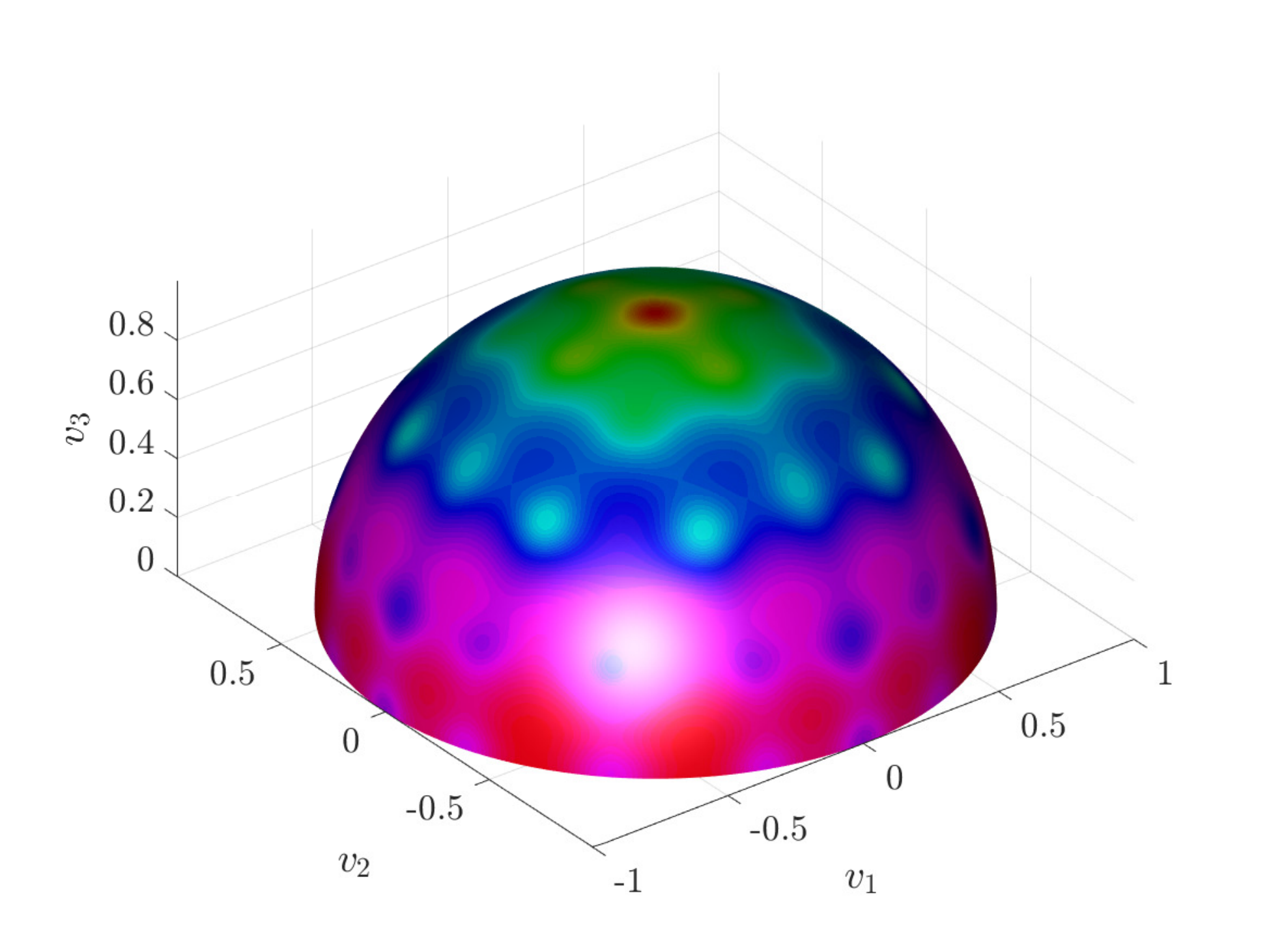}\,
	\includegraphics[scale=0.39]{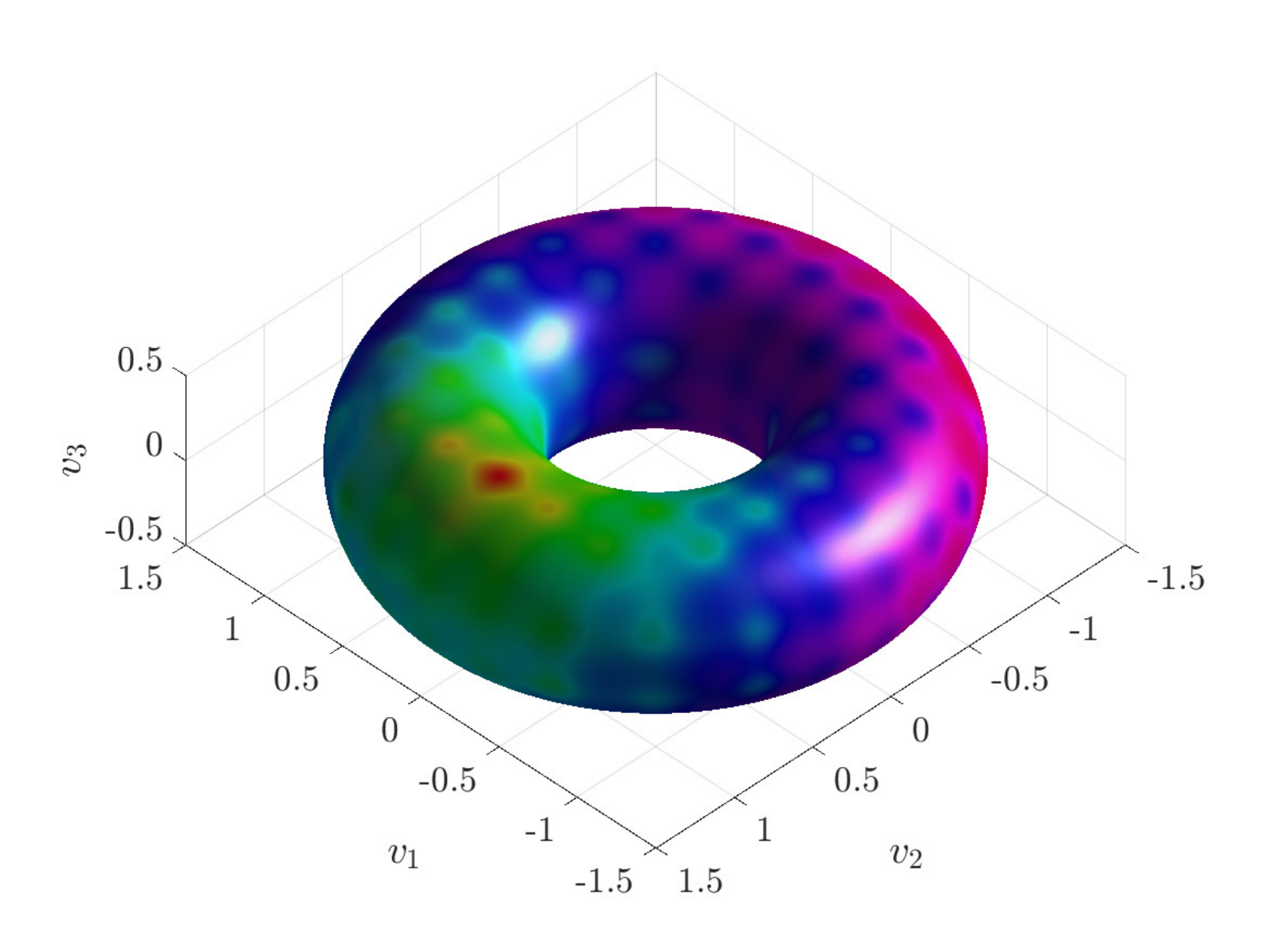}\,
	\caption{The Ackley function for $d=3$ in the constrained case on two example hypersurfaces. The half sphere $\SS^2$ (left) and the torus with $R=1$ and $r=0.5$ (right). The global minimum on $\SS^2$ corresponds to the point $v_*=(0,0,1)^T$, over the torus to the point $v_*=(0,1,0.5)^T$.}
	\label{fg:ackley}
\end{figure}

%\begin{figure}[tb]
%\includegraphics[scale=0.4]{ackley2.eps}\,
%\includegraphics[scale=0.4]{ackleys2.eps}\,
%\caption{The Ackley function for $d=3$ in spherical coordinates on $[0,\pi)^2$ (left) and its representation %over the half hypersurface $\SS^2$ (right). The global minimum is at $\theta_1^*=\theta_2^*=\pi/2$ which corresponds %to the direction $v^*=(0,0,1)^T$.}
%\label{fg:ackley}
%\end{figure}

\begin{figure}[htb]
	\hskip .2cm
	\includegraphics[scale=0.2]{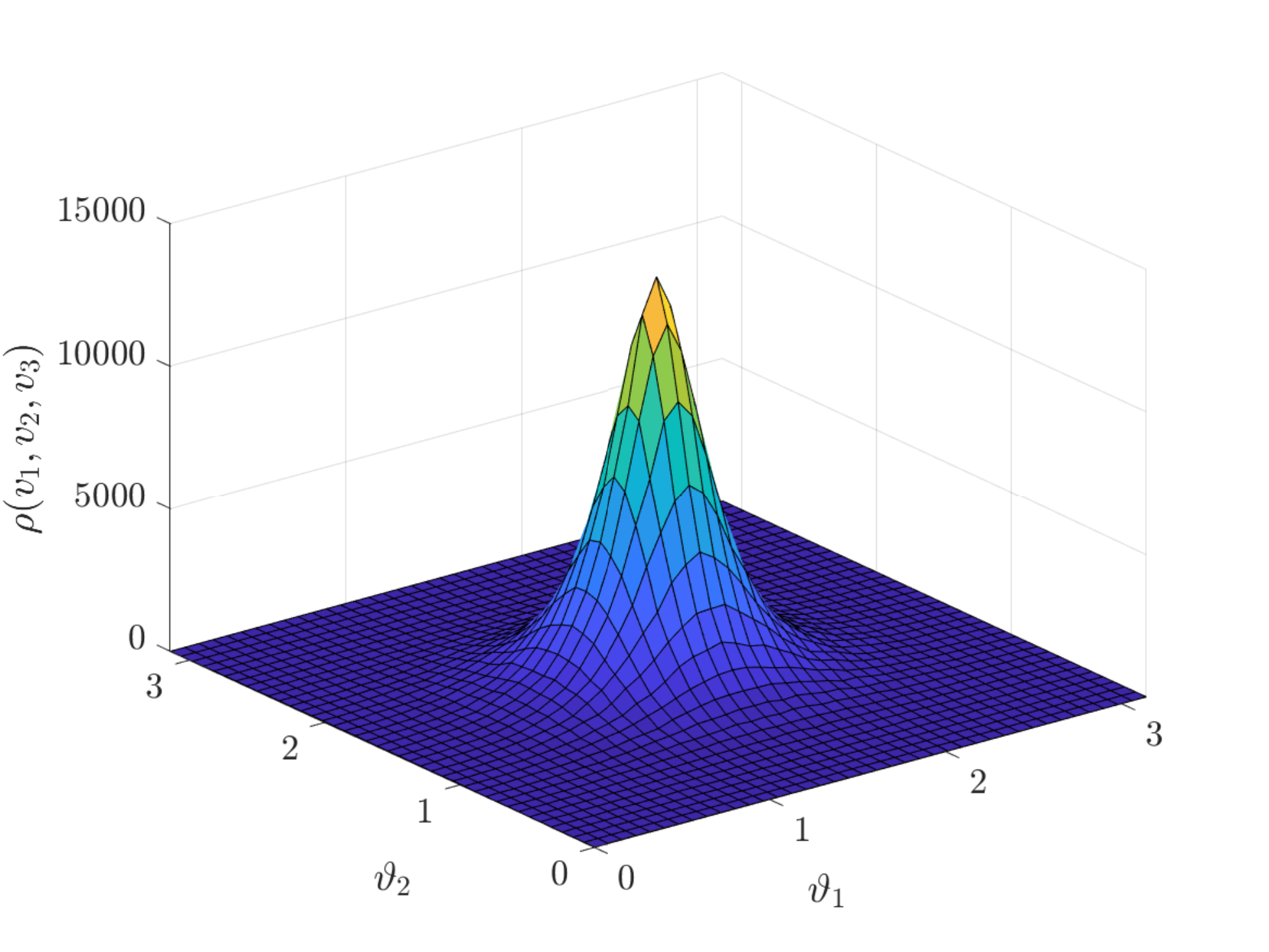}\hskip -.25cm
	\includegraphics[scale=0.2]{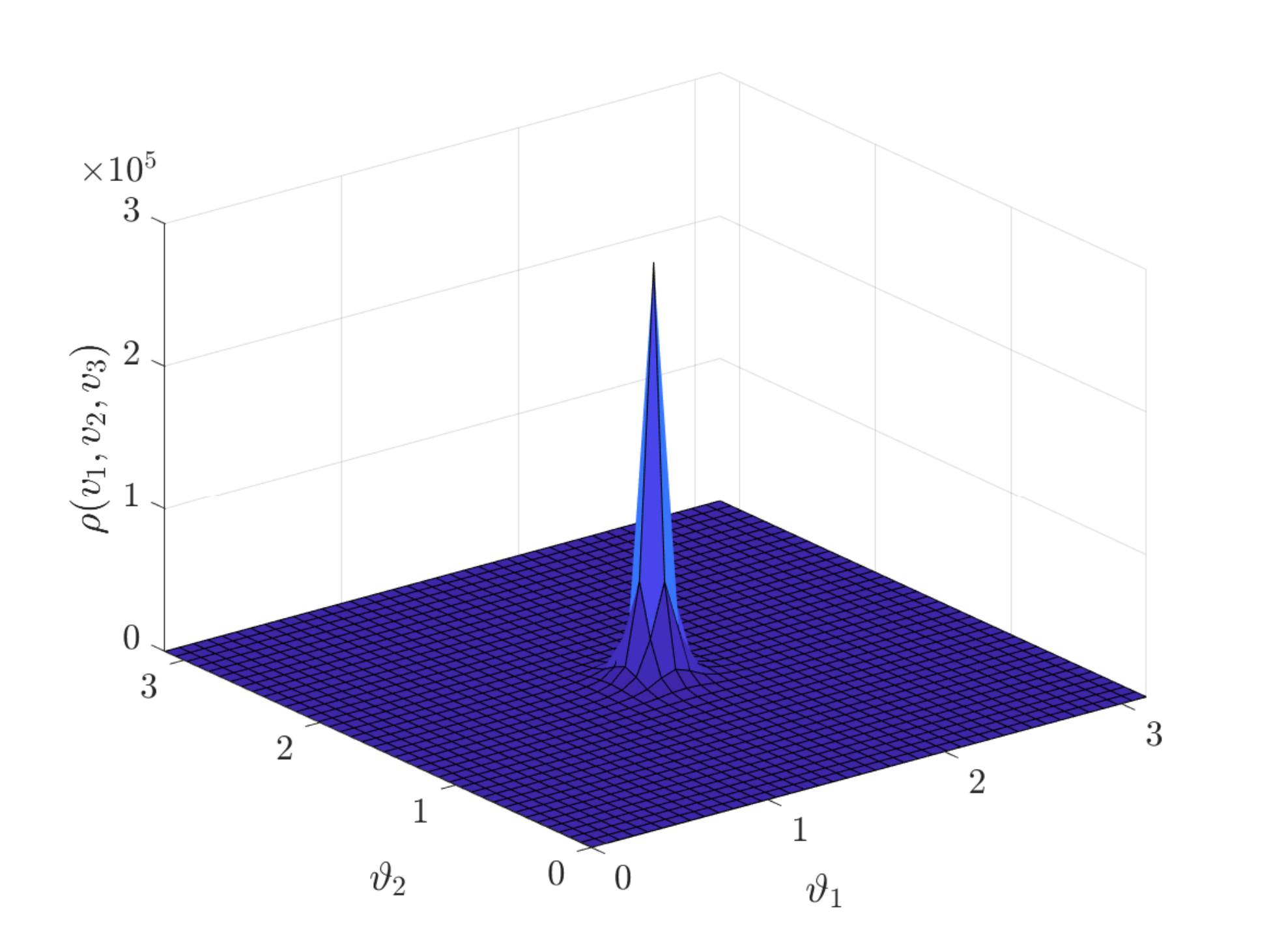}\hskip .4cm
	\includegraphics[scale=0.2]{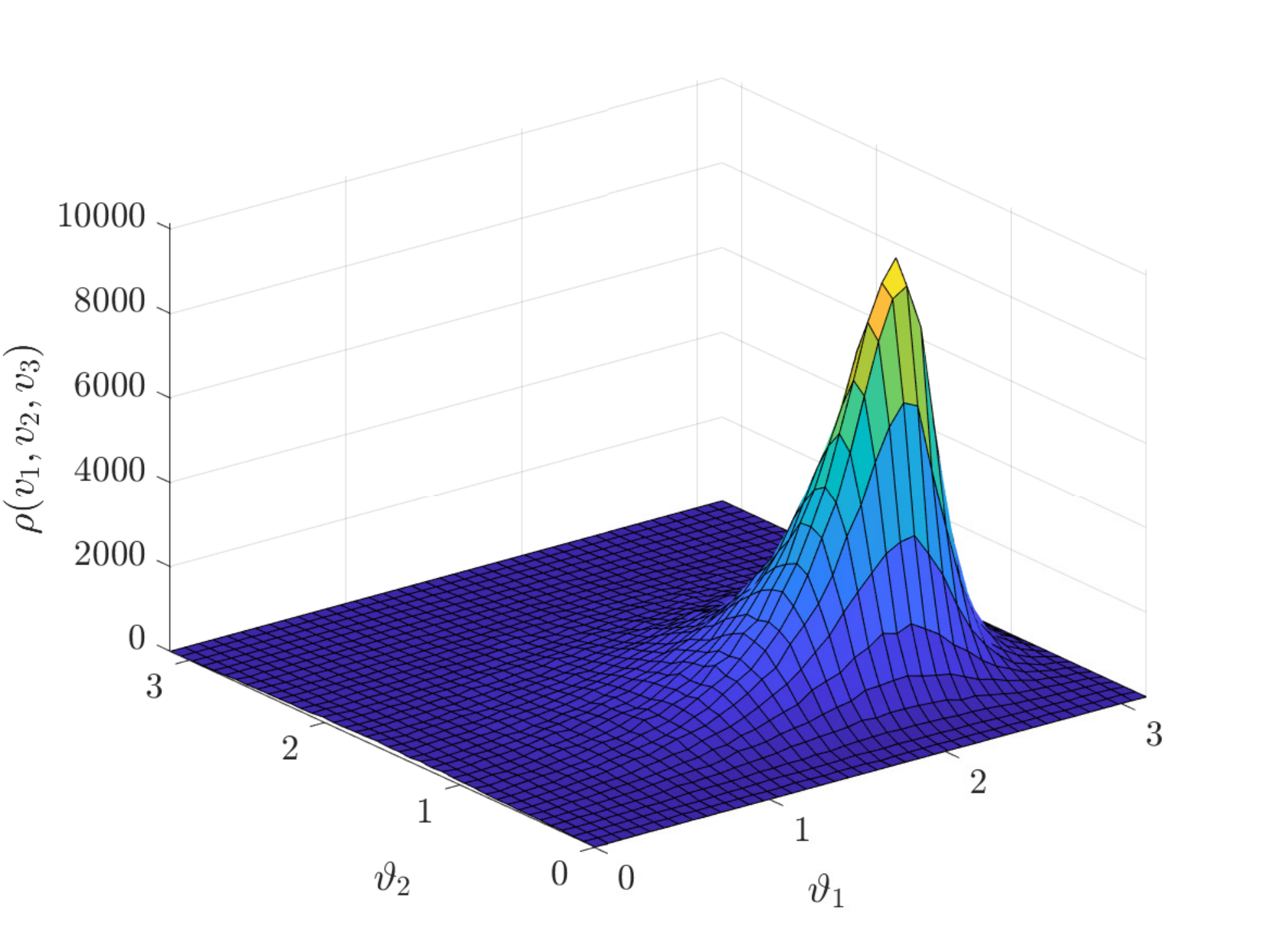}\hskip -.25cm
	\includegraphics[scale=0.2]{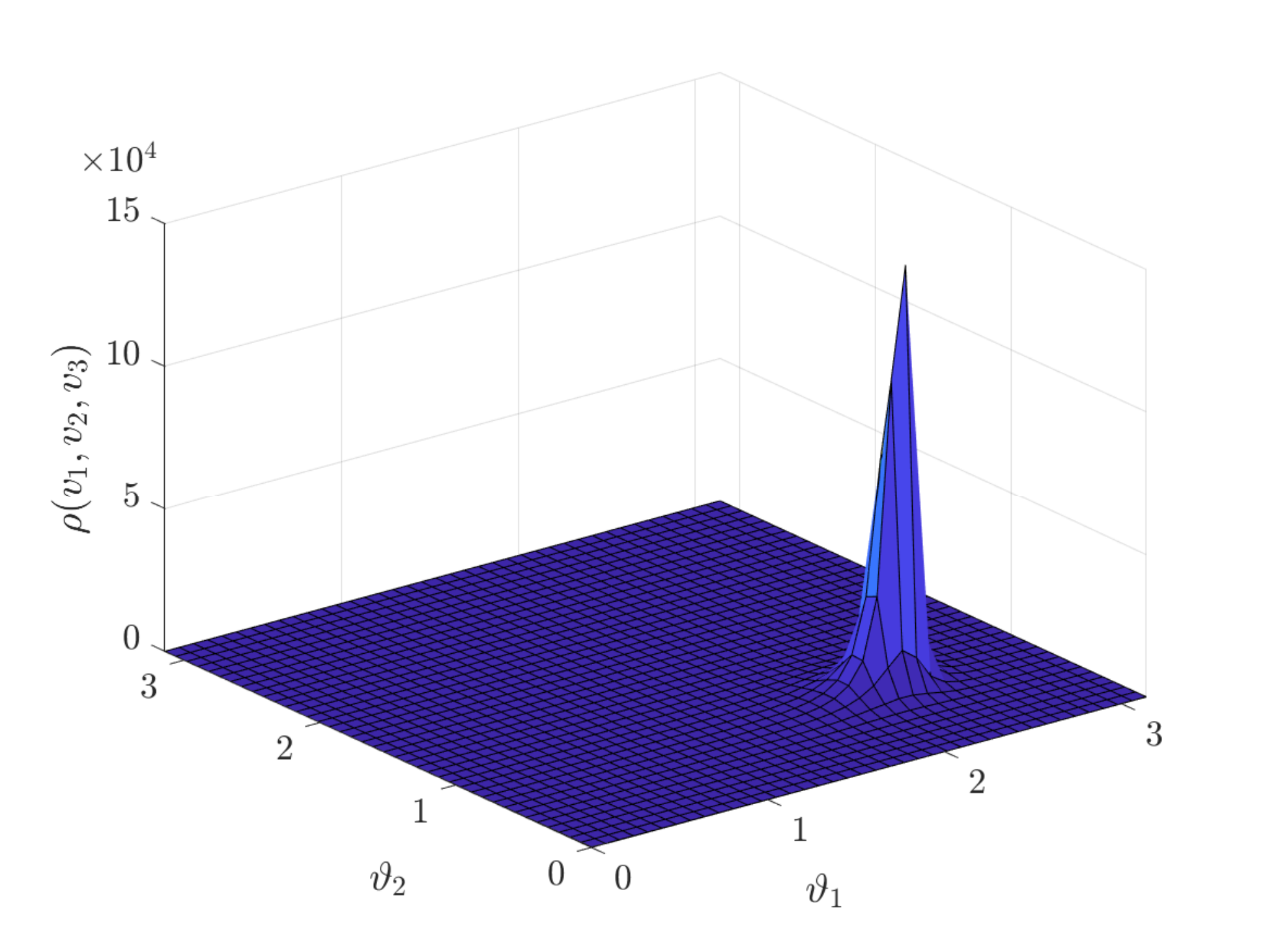}\\[-.2cm]
	\includegraphics[scale=0.39]{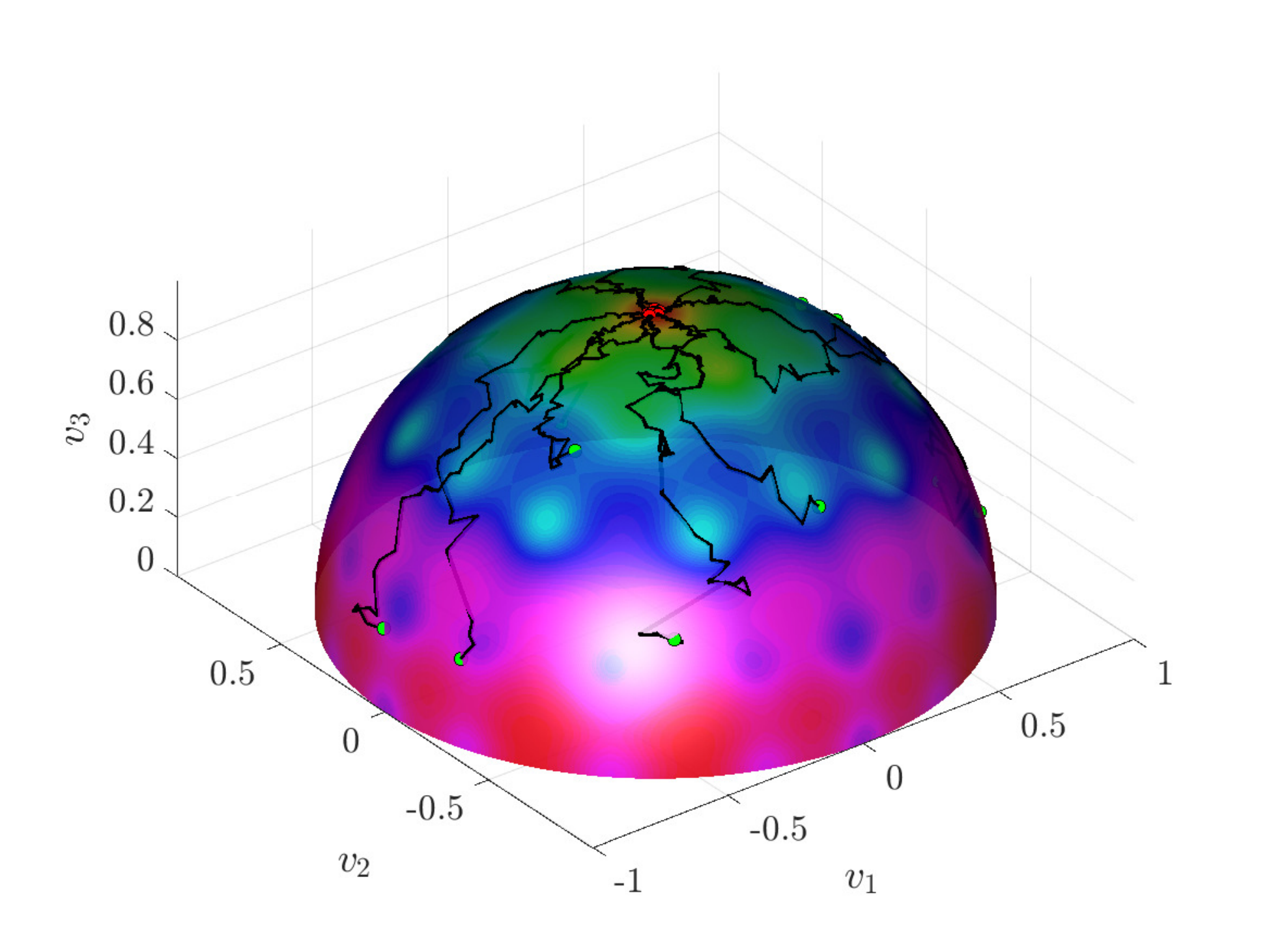}\,\hskip .2cm
	\includegraphics[scale=0.39]{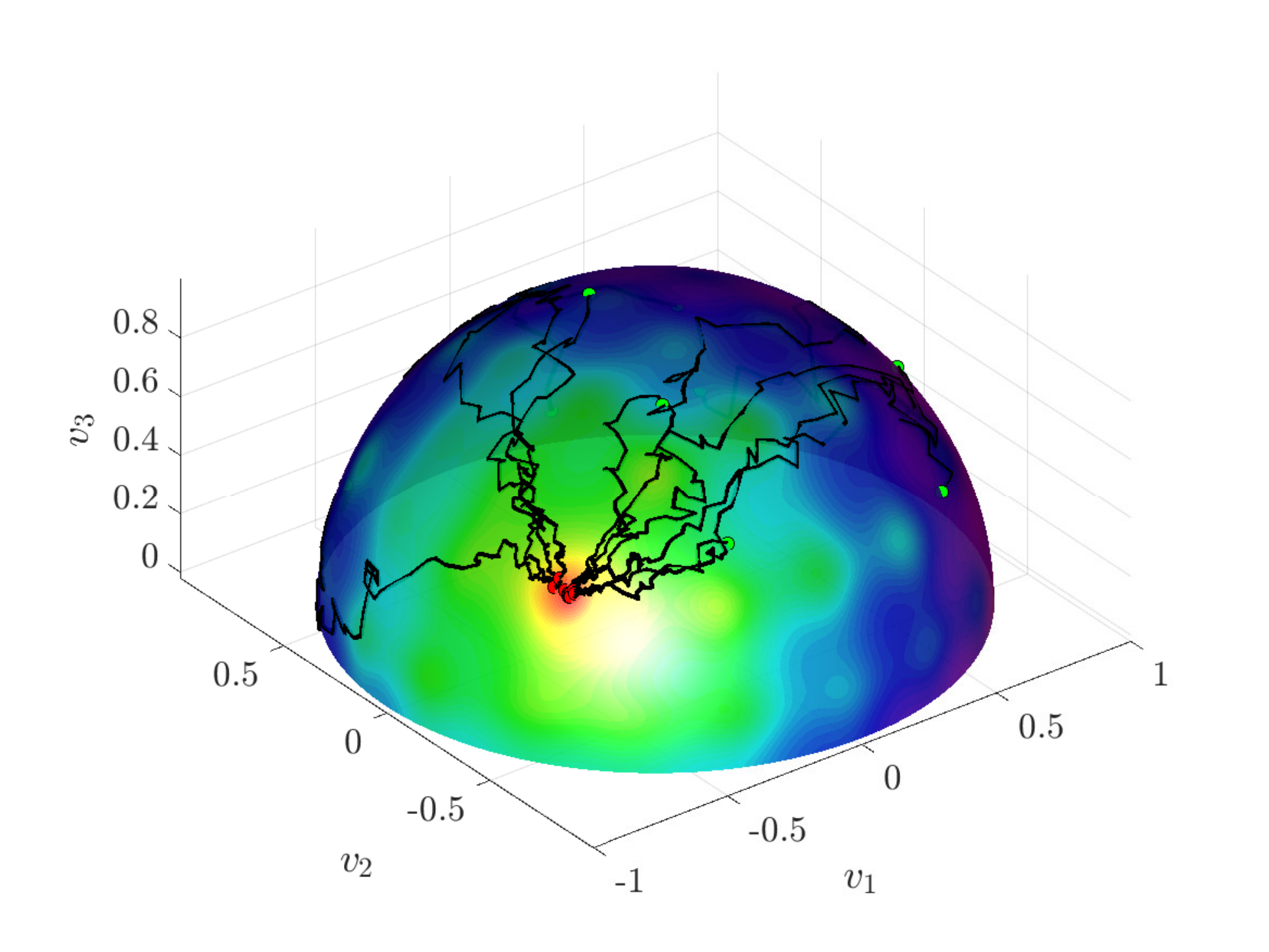}\,
	\caption{Particles trajectories along the simulation for the Ackley function on the half-sphere in the case $d=3$, $N=20$ with minimum at $v_*=(0,0,1)^T$ (left) and $v_*=(-1/\sqrt{2},-1/2,1/2)^T$ (right). On the top corresponding time evolution of the particle distribution $\rho(t,v)$ in angular coordinates at $t=1$ and $t=2.5$ for $N=10^6$. The simulation parameters are $\Delta t=0.05$, $\sigma=0.25$ and $\alpha=50$.  Although we run the simulation on the entire sphere, we display the result  on the half-sphere because, eventually, all particles do converge to one point, hence, all particles do eventually belong to a half sphere (mind that the stochastic process is a time continuous function). Displaying half-sphere is simply aesthetically more pleasing and it allows a better reading of the figures. On the other hand, as soon as one generates initial particles on the half-sphere containing a minimizer, then they - empirically - tend generically to remain on that half sphere, hence there is no need of displaying the rest of the sphere and we do not use any confinement mechanism rather than the given dynamics.  Also boundary conditions are actually not imposed. }
	\label{fg:ackley_traj}
\end{figure}

\begin{figure}[tb]
	\includegraphics[scale=0.39]{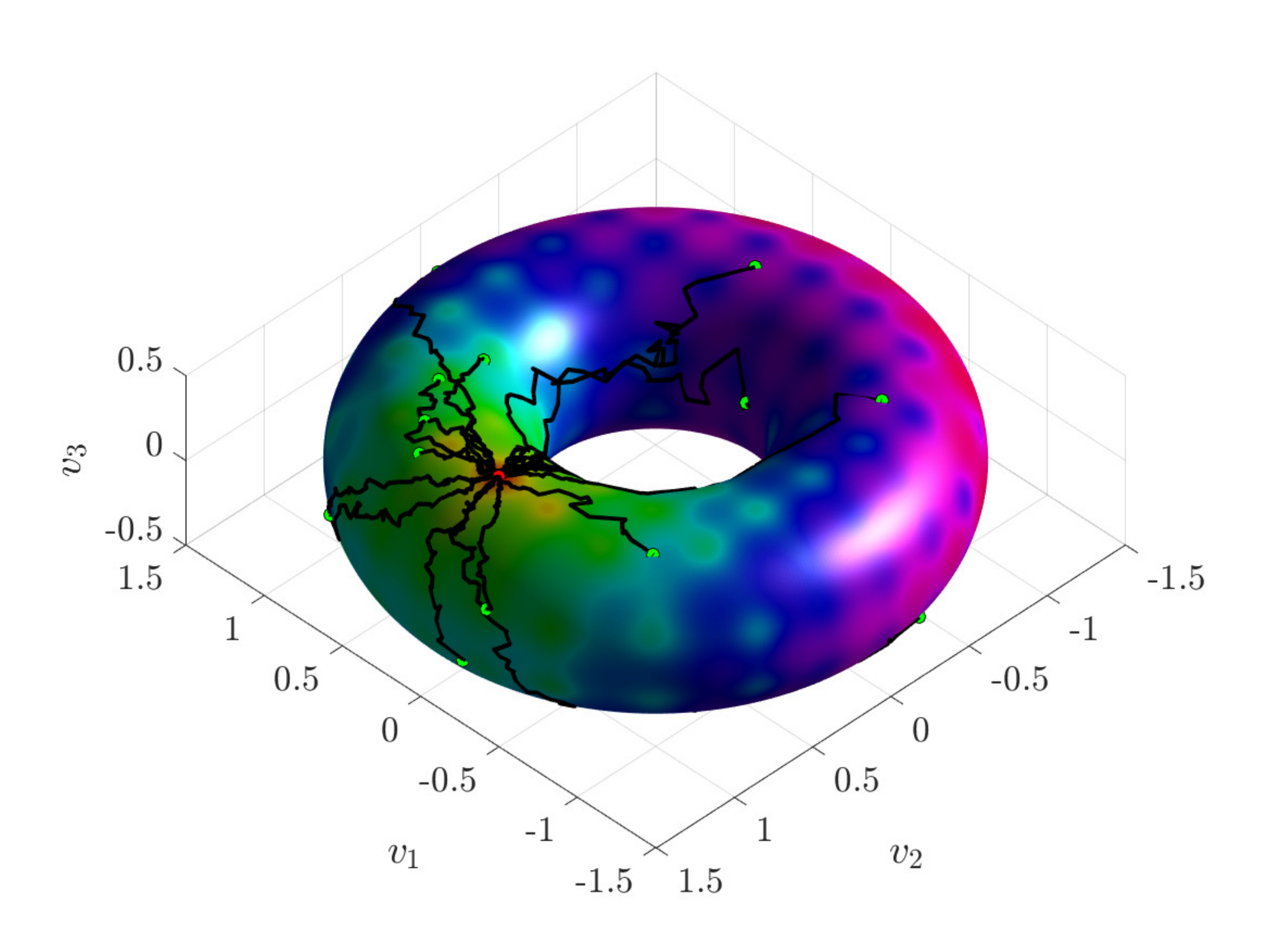}\,
	\includegraphics[scale=0.39]{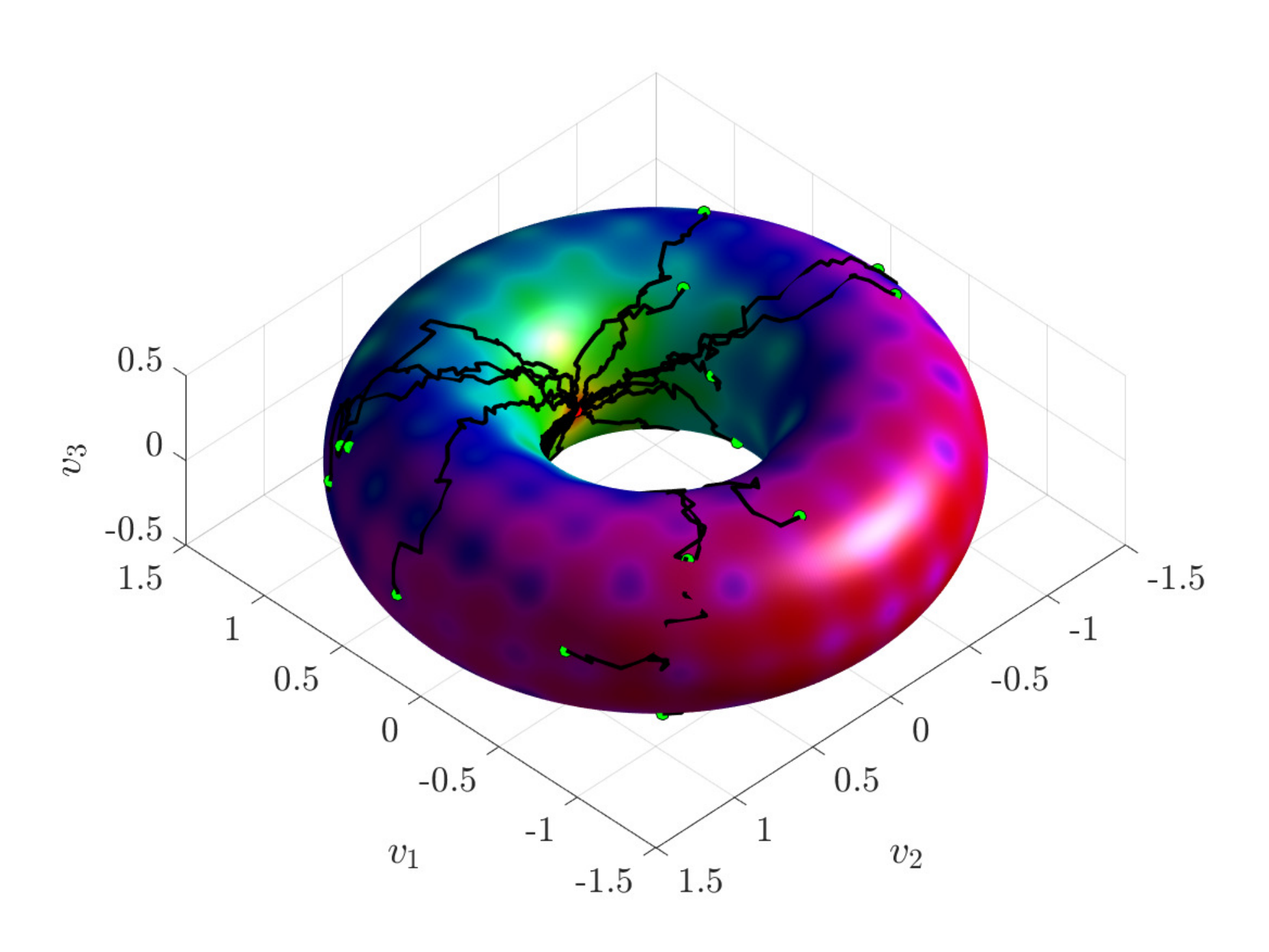}\,
	\caption{Particles trajectories along the simulation for the Ackley function constrained on the torus with $R=1$ and $r=0.5$ in the case $d=3$, $N=20$ with minimum at $v_*=(0,1,0.5)^T$ (left) and $v_*=(0.5,0,0)^T$ (right). The simulation parameters are $\Delta t=0.05$, $\sigma=0.25$ and $\alpha=50$. }
	\label{fg:ackleyt_traj}
\end{figure}

%\begin{figure}[htb]
%\hskip .2cm
%\includegraphics[scale=0.2]{rhot20.eps}\hskip -.25cm
%\includegraphics[scale=0.2]{rhot50.eps}\hskip .4cm
%\includegraphics[scale=0.2]{rhot20b.eps}\hskip -.25cm
%\includegraphics[scale=0.2]{rhot50b.eps}\\[-.2cm]
%\includegraphics[scale=0.41]{ackley_traj2.eps}\,\hskip .2cm
%\includegraphics[scale=0.41]{ackley_trajb2.eps}\,
%%\includegraphics[scale=0.25]{rhot20.eps}\,
%%\includegraphics[scale=0.25]{rhot50.eps}\\
%%\includegraphics[scale=0.4]{ackley_traj.eps}\\,
%%\includegraphics[scale=0.6]{ackley_trajb.eps}\,
%\caption{Particles trajectories along the simulation for the Ackley function in the case $d=3$, $N=20$ with minimum at $v^*=(0,0,1)^T$ (left) and $v^*=(-1/\sqrt{2},-1/2,1/2)^T$ (right). On the top corresponding time evolution of the particle distribution $\rho(t,v)$ in angular coordinates at $t=1$ and $t=2.5$ for $N=10^6$. The simulation parameters are $\Delta t=0.05$, $\sigma=0.25$ and $\alpha=50$.}
%\label{fg:ackley_traj}
%\end{figure}

In the reported simulations we initialize the particles with a uniform distribution over the hypersurface (this typically requires the use of a suitable algorithm, see Ref. \cite{fhps20-2}) 
%$\theta_{1},\theta_2 \in [0,\pi)$ 
and employ a simple Euler-Maruyama scheme with projection and time step $\Delta t>0$. We report in Figure \ref{fg:ackley_traj} the case of the half-sphere together with the particle trajectories for $t\in [0,5]$ with $N=20$, $\Delta t=0.05$, $\sigma=0.25$ and $\alpha=50$. On the left we consider the case with minimum at $v_*=(0,0,1)^T$, on the right the case with minimum at $v_*=(1/\sqrt{2},-1/2,1/2)^T$. The time evolution of the particle distribution $\rho(t,v)$ in the numerical mean field limit for $N=10^6$ is also reported in the upper part of the same figure.

In Figure \ref{fg:ackleyt_traj} we report an analogous simulation in the constrained case on the torus. The numerical parameters are the same as in the previous case. On the left the minimum is located at $v_*=(0,1,0.5)^T$, whereas on the right is located at $v_*=(0.5,0,0)^T$.

%We report in Figure \ref{fg:ackley_traj} the particle trajectories for $t\in [0,5]$ in the case of $N=20$, \Delta t=0.05$, $\sigma=0.25$ and $\alpha=50$. On the left we consider the case with minimum at $\phi_k=\pi/2$, corresponding to $v^*=(0,0,1)^T$, on the right the case with minimum at $\phi_k=\pi/4$, corresponding to $v^*=(1/\sqrt{2},-1/2,1/2)^T$. The time evolution of the particle distribution $\rho(t,v)$ in the numerical mean-field limit for $N=10^6$ is also reported in the upper part of the same figure. 

As these figures show it, the dynamics consistently converges to the global minimizers totally regardless of the many local minimizers and the large particle system does converges numerically to its mean-field approximation.  In the related paper Ref. \cite{fhps20-2}  we rigorously analyze the large time behavior of the system for  $\Gamma = \mathbb S^{d-1}$, we prove its convergence to global minimizers in a suitable sense, and we provide several applications in machine learning. A proof of convergence to global minimizers for more general hypersurfaces is the subject of current investigation. 
\\

Let us conclude this introduction by mentioning that the optimization on the hypersurface offers further numerous advantages, besides allowing a rigorous proof of mean-field limit  thanks to the compactness of the hypersurface.  First of all a vast class of optimization problems can be reduced to constrained optimizations over  hypersurfaces: in a related paper Ref. \cite{fhps20-2}, where we focus on the numerical implementation of the method on the sphere and its asymptotic/convergence analysis  to global minimizers, we present also a few relevant and challenging applications in signal processing and machine learning. 
%Moreover, the concentration of measure phenomenon on the hypersurface \cite{leta91} tends to favor convergence of CMO;  we expect the convergence not be affected by the curse of dimension, see Remark \ref{cursedim} below. 
Due to  compactness of the hypersurface, local smoothness and boundedness requirements on $\EE$ are necessarily a uniform and global property. However, against these properties that greatly simplify the analysis of the well-posedness of the system and its mean-field limit, the specific topology of the hypersurface may make 
it  harder to prove asymptotic convergence of the dynamics to global miminizers, requiring major technical variations with respect to the approach of unconstrained CBO \cite{carrillo2018analytical}. We refer to Ref. \cite{fhps20-2} for the details of the large time analysis of solutions $\rho(t,v)$ of \eqref{PDE} on the sphere and examples of applications in high-dimension. There we needed to develop {\it ad hoc} arguments to prove the asymptotic convergence.  A similar consensus dynamics could be in principle defined also on more general compact manifolds \cite{emery1989stochastic}, but we still lack a general approach to deal with its (global) asymptotic analysis and this research direction needs certainly further investigations. 
\\

The rest of the paper is organized as follows: in Section \ref{sec:wellposedness} we  show that the equations \eqref{stochastic Kuramoto-Vicsek}, \eqref{PDE} and \eqref{selfprocess} are well-posed, that is, there is a unique solution for any initial distribution $\rho_0$, which depends continuously on the initial datum, and in Section \ref{sec:MFlimit} we address the mean-field limit.

\section{Well-posedness}\label{sec:wellposedness}
This section focuses on proving the well-posedness for the particle system \eqref{stochastic Kuramoto-Vicsek}, the mean-field dynamic \eqref{selfprocess} and the mean-field PDE \eqref{PDE}. To ensure this, we make the following smoothness assumption on the objective function $\EE$.
\begin{assu}\label{asum1}
	The objective function $0\leq \EE:\mathbb R^{d}\rightarrow \RR$ is locally Lipschitz continuous.
\end{assu}

\subsection{Well-posedness for the interacting particle system \eqref{stochastic Kuramoto-Vicsek}}
Recall that the particle system \eqref{stochastic Kuramoto-Vicsek} is assumed to evolve in the whole space $\RR^d$ instead of on the hypersurface $\Gamma$ directly. We choose this embedding because it provides an explicit and computable representation of the system and it allows for a global description.
The difficulty in showing first  the well-posedness of \eqref{stochastic Kuramoto-Vicsek} in the ambient space $\RR^d$  is that the projection $P(V_t^{i})$, $\Delta \gamma (V_t^{i})$ and $\nabla\gamma (V_t^{i})$ may not be well defined when $V_t^{i}\notin \widehat\Gamma$.  For example, in the case of the sphere $\SS^{d-1}$, the projection $P(V_t^{i})=I-\frac{V_t^{i}(V_t^{i})^T}{|V_t^{i}|^2}$ and $\Delta\gamma (V_t^{i})\nabla\gamma (V_t^{i})=(d-1)\frac{V_t^{i}}{|V_t^{i}|^2}$  are not defined for $V_t^{i}=0$.
In order to overcome this problem, we regularize the diffusion and drift coefficients, that is, we replace them with appropriate functions $P_1$, $P_2$ and $P_3$ respectively: let $P_1$ be a $d\times d$ matrix valued map on $\RR^d$ with bounded derivatives such that $P_1(v)= P(v)$ for all $v\in\widehat\Gamma$, $P_2$ be a $\RR^d$ valued map on $\RR$ with bounded derivatives such that 
$P_2(v)=\Delta\gamma(v)$ if $v\in \OG$, and $P_3$ be a $\RR^d$ valued map on $\RR^d$, again with bounded derivatives such that 
$P_3(v)=\nabla\gamma(v)$ if $v\in \OG$.   It is also useful to mention that
\begin{equation}
P(v)\nabla\gamma(v)=0
\end{equation}
and
\begin{equation}\label{P2}
\nabla\gamma(v)\cdot P(v)y=0
\end{equation}
hold for any $y\in \RR^d$. For later use, let us denote $[N] = \{1,\dots,N\}$. 
Additionally, we regularize the locally Lipschitz function $\EE$: Let us introduce $\TE(v)$ satisfying the following assumptions 
\begin{assu}\label{asum}
	The regularized extension function $\TE:\RR^d\rightarrow \RR$ is globally Lipschitz continuous and satisfies the properties
	\begin{itemize}
		\item[1.]  $\TE(v)=\EE(v)$ when $v\in \OG$;
		\item[2.]  $\TE(v)-\TE(u)\leq L|v-u|$ for all $u,v\in\RR^d$ for a suitable global Lipschitz constant $L>0$;
		\item[3.]	$-\infty < \underline{\TE}:=\inf \TE \leq \TE\leq \sup \TE=:\overline \TE < + \infty$\,.
	\end{itemize}
\end{assu}
\begin{remark}
	Here $\TE$ is introduced as an auxiliary function for the proof of well-posedness and mean-field limit only, and it does not play any role in actual optimization problem, which is defined on $\Gamma$. Indeed,
	as we can see in Theorem \ref{thmwellposednessofparticle} below, particles stay on the hypersurface $\Gamma$ all the time, which means that certainly $v\in \OG$, so one has $\TE(v)\equiv\EE(v)$. 
\end{remark}

We note here that $\EE$ and $\TE$ appearing the sequel always satisfy Assumptions \ref{asum1} and \ref{asum}. 
Given such $P_1$, $P_2$, $P_3$ and $\TE$, we introduce the following regularized particle system
\begin{align} \label{RSKV}
dV_t^i &= -\lambda P_1(V_t^i)(V_t^i-v_{\alpha,\TE}(\rho_t^N))dt + \sigma |V_t^i - v_{\alpha,\TE}(\rho_t^N)| P_1(V_t^i)dB_t^i-\frac{\sigma^2}{2}(V_t^i-v_{\alpha,\TE}(\rho_t^N))^2P_2(V_t^i)P_3(V_t^i)dt\,,
\end{align}
for $i \in [N]$,  where 
\begin{equation}\label{ValphaERd}
v_{\alpha,\TE}(\rho_t^N)=\frac{\int_{\mathbb R^{d}}v\omega_\alpha^{\TE} (v)d\rho_t^N}{\int_{\mathbb R^{d}}\omega_\alpha^{\TE}(v)d\rho_t^N}\,, \qquad  \omega_\alpha^{\TE} (v)=e^{-\alpha{\TE} (v)}\,.
\end{equation}

Next we can deduce that the coefficients in \eqref{RSKV}  are locally Lipschitz continuous and have linear growth. More precisely, we have the following result.

\begin{lemma}\label{lemlocal}
	Let $N\in \mathbb{N}$, $\alpha>0$ be arbitrary and $\TE$ satisfy Assumption \ref{asum}. Then for any $\mb{V}^N,\widehat{\mb{V}}^N\in \RR^{Nd}$, 
	and corresponding empirical measures $\rho^N=\frac{1}{N}\sum_{i=1}^{N}\delta_{V^i}\,,$ and $\widehat \rho^N=\frac{1}{N}\sum_{i=1}^{N}\delta_{\widehat V^i}\,,$
	it holds
	\begin{equation}
	|v_{\alpha,\TE}(\rho^N)|\leq  \frac{1}{N}C_{\alpha,\TE}\|\mb{V}^N\|_1
	\end{equation}
	and
	\begin{equation}
	|v_{\alpha,\TE}(\widehat \rho^N)-v_{\alpha,\TE}(\rho^N)|\leq \left(\frac{C_{\alpha,\TE}}{N}+\frac{2\alpha LC_{\alpha,\TE}}{N}\|\mb{\widehat V}^N\|_\infty\right)\|\mb{V}^N- \mb{\widehat V}^N\|_1\,,
	\end{equation}
	where $C_{\alpha,\TE}=e^{\alpha(\overline \TE-\underline \TE)}$. Here we used the notations for norms of vectors $\|\mb{V}\|_\infty=\sup_{i \in [N]}|V^i|$ and $\|\mb{V}\|_1=\sum_{i=1}^N|V^i|$.
\end{lemma}

\begin{proof}
	First, one notes that
	\begin{equation}
	e^{-\alpha \overline \TE}\leq \omega_{\alpha}^{\TE}(V^j)=e^{-\alpha \TE( V^j )}\leq  e^{-\alpha \underline \TE}\,.
	\end{equation}
	Then we have
	\begin{align*}
	|v_{\alpha,\TE}(\rho^N)|=\left| \frac{\sum_{j=1}^{N} V^j e^{-\alpha \TE( V^j )}}{\sum_{j=1}^{N}e^{-\alpha \TE( V^j )}}\right|\leq \frac{1}{N}e^{\alpha( \overline \TE-\underline \TE)} \sum_{j=1}^N | V^j | = \frac{1}{N}C_{\alpha,\TE} \|\mb{V}^N\|_1\,,\end{align*}
	where  $C_{\alpha,\TE} :=e^{\alpha( \overline \TE-\underline \TE)}$. Next we split the error
	\begin{align}
	v_{\alpha,\TE}(\widehat \rho^N)-v_{\alpha,\TE}(\rho^N)&=\frac{\sum_{j=1}^{N}( V^j - \widehat V^j )\omega_{\alpha}^{\TE}(V^j)}{\sum_{j=1}^{N}\omega_{\alpha}^{\TE}(V^j)}+\frac{\sum_{j=1}^{N}\widehat{V}^j(\omega_{\alpha}^{\TE}(V^j)-\omega_{\alpha}^{\TE}(\widehat V^j))}{\sum_{j=1}^{N}\omega_{\alpha}^{\TE}(V^j)}\notag\\
	&\quad +\frac{\sum_{j=1}^{N} \widehat V^j \omega_{\alpha}^{\TE}(\widehat V^j)\sum_{j=1}^{N}(\omega_{\alpha}^{\TE}(\widehat V^j)-\omega_{\alpha}^{\TE}( V^j))}{(\sum_{j=1}^{N}\omega_{\alpha}^{\TE}(V^j))(\sum_{j=1}^{N}\omega_{\alpha}^{\TE}(\widehat V^j))}\notag\\
	&=:I_1+I_2+I_3\,.
	\end{align}
	For $I_1$, the previous computation yields
	\begin{equation}\label{I1}
	|I_1|\leq \frac{1}{N}C_{\alpha,\TE}\|\mb{V}^N- \mb{\widehat V}^N\|_1\,,
	\end{equation}
	Notice that it holds
	\begin{align*}
	|\omega_{\alpha}^{\TE}(V^j)-\omega_{\alpha}^{\TE}(\widehat V^j)|=|e^{-\alpha \TE( V^j )}-e^{-\alpha \TE( \widehat V^j )}|\leq  \alpha e^{-\alpha \underline \TE}   |\TE( V^j )-\TE( \widehat V^j )|\leq \alpha Le^{-\alpha \underline \TE}|V^j-\widehat V^j|\,.
	\end{align*}
	Thus we conclude that
	\begin{equation}\label{I2}
	\max \{|I_2|\,,|I_3|\}\leq  \frac{\alpha LC_{\alpha,\TE}}{N}\|\mb{\widehat V}^N\|_\infty\|\mb{V}^N- \mb{\widehat V}^N\|_1\,.
	\end{equation}
	Collecting estimates \eqref{I1} and \eqref{I2} completes the proof.
\end{proof}

\begin{theorem}\label{thmwellposednessofparticle}
	Under Assumptions \ref{asum1} and \ref{asum}  let $\rho_0$ be a probability measure on $\Gamma$ and, for every $N\in\mathbb{N}$, $(V_0^i)_{i \in [N]}$ be $N$ i.i.d. random variables with the common law $\rho_0$.
	For every $N\in\mathbb{N}$, there exists a pathwise unique strong solution $((V_t^i)_{t\geq 0})_{i \in [N]}$ to the particle system \eqref{stochastic Kuramoto-Vicsek} with the initial data $(V_0^i)_{i \in [N]}$. Moreover it holds that $V_t^i\in \Gamma$ for all $i \in [N]$ and any $t>0$.
\end{theorem}
\begin{proof}
	Given $P_1$, $P_2$, $P_3$ and $\TE$, the SDE \eqref{RSKV} has locally Lipschitz coefficients, so it admits a pathwise unique local strong solution by standard SDE well-posedness result \cite[Chap. 5, Theorem 3.1]{durrett2018stochastic}.  Moreover, it follows from It\^{o}'s formula \eqref{Mulito} (choosing $\varphi(x)=\gamma(x)$), as long as $V_t^i\in \OG$ (the process is continuous and we can nevertheless consider a smooth extension of $\gamma$ out of $\OG$), that 
	\begin{align}\label{dVti2=0}
	d\gamma(V_t^i)&= -\lambda \nabla\gamma(V_t^i)\cdot P(V_t^i)(V_t^i-v_{\alpha,\TE}(\rho_t^N))dt + \sigma |V_t^i - v_{\alpha,\TE}(\rho_t^N)| \nabla\gamma(V_t^i)\cdot P(V_t^i)dB_t^i \notag\\
	&\quad -\frac{1}{2}\sigma^2(V_t^i-v_{\alpha,\TE}(\rho_t^N))^2\Delta \gamma(V_t^i)dt\notag\\
	&\quad+\frac{1}{2}\sigma^2(V_t^i-v_{\alpha,\TE}(\rho_t^N))^2\sum_{k,\ell=1}^d\partial_{k,\ell}^2\gamma(V_t^i)\left(e_k^T-\partial_k\gamma(V_t^i)\nabla\gamma(V_t^i)^T\right)\left(e_\ell-\partial_{\ell}\gamma(V_t^i)\nabla\gamma(V_t^i)\right)dt\notag\\
	&=-\frac{1}{2}\sigma^2(V_t^i-v_{\alpha,\TE}(\rho_t^N))^2\Delta \gamma(V_t^i)dt+\frac{1}{2}\sigma^2(V_t^i-v_{\alpha,\TE}(\rho_t^N))^2\sum_{k}^d\partial_{k}^2\gamma(V_t^i) dt\notag\\
	&\quad +\frac{1}{2}\sigma^2(V_t^i-v_{\alpha,\TE}(\rho_t^N))^2\left(-\sum_{k,\ell=1}^d\partial_{k,\ell}^2\gamma(V_t^i)\partial_{k}\gamma(V_t^i)\partial_{\ell}\gamma(V_t^i)\right)dt \notag\\
	&= -\frac{1}{2}\sigma^2(V_t^i-v_{\alpha,\TE}(\rho_t^N))^2\Delta \gamma(V_t^i)dt
	+\frac{1}{2}\sigma^2(V_t^i-v_{\alpha,\TE}(\rho_t^N))^2\Delta \gamma(V_t^i)dt\notag\\
	&\quad +\frac{1}{2}\sigma^2(V_t^i-v_{\alpha,\TE}(\rho_t^N))^2\left(-\frac{1}{2}\sum_{\ell}^d\partial_{\ell}|\nabla\gamma(V_t^i)|^2\partial_{\ell}\gamma(V_t^i)\right)dt\quad (\mbox{ since }|\nabla\gamma|=1)\notag\\
	&=0
	\,, 
	\end{align}
	where $\partial_k\gamma(V_t^i)$ is the $k$-th component of $\nabla\gamma(V_t^i)$, $e_k$ is the unit vector along the $k$th dimension, and  we have used the property of $P(V_t^i)$ as in \eqref{P2}. Hence $\gamma(V_t^i)=\gamma(V_0^i)=0$ for all $t>0$, which ensures that the solution keeps bounded at finite time, hence we have a global solution. Since all $V_t^i\in\Gamma$, the solution to the regularized system \eqref{RSKV} is a solution to \eqref{stochastic Kuramoto-Vicsek}, which provides the global existence of solutions to \eqref{stochastic Kuramoto-Vicsek}. 
	
	To show pathwise uniqueness let us consider two solutions to \eqref{stochastic Kuramoto-Vicsek} for the same initial distribution and Brownian motion. According to the above computation these two solutions stay on the hypersurface for any $t\geq 0$, hence they are solutions to the regularized system \eqref{RSKV}, whose solutions are pathwise unique due to the locally Lipschitz continuous coefficients.  Hence we have uniqueness for solutions to \eqref{stochastic Kuramoto-Vicsek}.
\end{proof}

\subsection{Well-posedness for the mean-field dynamic \eqref{selfprocess}}
For readers' convenience, we give a brief introduction of the Wasserstein metric in the following definition, we refer to Ref. \cite{ambrosio2008gradient} for more details.
\begin{definition}[Wasserstein Metric]
	Let $1\leq p < \infty$ and $\mc{P}_p(\RR^{d})$ be the space of Borel probability measures on $\RR^{d}$ with finite $p$-moment. We equip this space with the Wasserstein distance 
	\begin{equation}
	W_p^{p}(\mu, \nu):=\inf\left\{\int_{\RR^d\times \RR^d} |z-\hat{z}|^{p}\ d\pi(\mu, \nu)\ \big| \ \pi \in \Pi(\mu, \nu)\right\}
	\end{equation}
	where $\Pi(\mu, \nu)$ denotes the collection of all Borel probability measures on $\RR^d\times \RR^d$ with marginals $\mu$ and $\nu$ in the first and second component respectively. The Wasserstein distance can also be expressed as
	\begin{equation}
	W_p^{p}(\mu, \nu) = \inf \left\{\mathbb{E}[|Z-\overline {Z}|^{p}]\right\}
	\end{equation}
	where the infimum is taken over all joint distributions of the random variables $Z$, $\overline {Z}$ with marginals $\mu$, $\nu$ respectively.
\end{definition}
Notice now that, for any $\rho \in \mathcal P_1(\mathbb R^d)$
\begin{align}
\frac{e^{-\alpha \underline \TE}}{\|{e^{-\alpha \TE}}\|_{L^1(\rho)}}\leq e^{\alpha(\overline \TE-\underline \TE)}=:C_{\alpha,\TE}\,.
\end{align}
A direct application of above leads to
\begin{equation}\label{esv}
|v_{\alpha,\TE}(\rho)|:= \left|\frac{\int_{\RR^{d}}  v  \omega_{\alpha}^{\TE}(v)\,d\rho}{\int_{\RR^{d}}\omega_{\alpha}^{\TE}(v)\,d\rho}\right|= \left|\frac{\int_{\RR^{d}}  v  e^{-\alpha\TE( v )}\,d\rho}{\|{e^{-\alpha \TE}}\|_{L^1(\rho)}}\right|\leq  C_{\alpha,\TE}\int_{\RR^d}|v|d\rho\,.
\end{equation}

\begin{lemma}\label{lemsta} Assume that $\rho,\widehat \rho\in\mc{P}_c(\RR^d)$ (with compact support), then the following stability estimate holds
	\begin{equation}\label{lemstaeq}
	|v_{\alpha,\TE}(\rho)-v_{\alpha,\TE}(\widehat \rho)|\leq CW_p(\rho,\widehat \rho)\,,
	\end{equation}
	for any $1\leq p<\infty$, where $C= C( C_{\alpha,\TE},\alpha,L)>0$. \end{lemma}
\begin{proof}
	Let us compute the difference
	\begin{align}
	v_{\alpha,\TE}(\rho)-v_{\alpha,\TE}(\widehat \rho)&=\frac{\int_{\RR^{d}}  v  e^{-\alpha \TE( v )}\,d\rho(v)}{\|{e^{-\alpha \TE}}\|_{L^1(\rho)}}-\frac{\int_{\RR^{d}}  \widehat v  e^{-\alpha \TE( \widehat v )}\,d\widehat \rho(\widehat v)}{\|{e^{-\alpha \TE }}\|_{L^1(\widehat \rho)}} \notag\\
	&=\iint_{\RR^{d}\times \RR^{d}} \underbrace{\frac{ v e^{-\alpha \TE( v )}}{\|{e^{-\alpha \TE}}\|_{L^1(\rho)}} -\frac{ \widehat v e^{-\alpha \TE( \widehat v )}}{\|{e^{-\alpha \TE }}\|_{L^1(\widehat \rho)}}}_{:=h(v)-h(\widehat{v})} d\pi(v,\widehat v) \notag
	\end{align}
	where $\pi \in \Pi(\rho, \widehat{\rho})$ is an arbitrary coupling of $\rho$ and $\widehat \rho$. We can write the integrand as follows
	\begin{align*}
	h(v)-h(\widehat v)&= \underbrace{\frac{( v - \widehat v )e^{-\alpha \TE( \widehat v )}}{\|{e^{-\alpha \TE }}\|_{L^1(\rho)}}}_{=:I_1} +\underbrace{\frac{ \widehat v (e^{-\alpha \TE( v )}-e^{-\alpha \TE( \widehat v )})}{\|{e^{-\alpha \TE}}\|_{L^1(\rho)}}}_{=:I_2} +  \underbrace{\frac{\widehat{v}e^{-\alpha \TE( \widehat v )}}{\|e^{-\alpha \TE }\|_{L^{1}(\rho)}} - \frac{ \widehat v e^{-\alpha \TE( \widehat v )}}{\|e^{-\alpha \TE}\|_{L^{1}(\widehat \rho)}}}_{:=I_3}
	\end{align*}
	where the last two terms on the right hand side can be rewritten as
	\begin{align*}	
	&\frac{ \widehat v e^{-\alpha \TE( \widehat v )}}{\|e^{-\alpha \TE}\|_{L^{1}(\rho)}} - \frac{ \widehat v e^{-\alpha \TE( \widehat v )}}{\|e^{-\alpha \TE }\|_{L^{1}(\widehat \rho)}}\notag\\
	&=  \frac{ \widehat v e^{-\alpha \TE( \widehat v )}\int e^{-\alpha \TE( \widehat v )}\ d\widehat{\rho}(\widehat{v})}{\|e^{-\alpha \TE }\|_{L^{1}(\rho)}\|e^{-\alpha \TE }\|_{L^1(\widehat{\rho})}} - \frac{ \widehat v e^{-\alpha \TE( \widehat v )}\int e^{-\alpha \TE( v )}\ d\rho(v)}{\|e^{-\alpha \TE }\|_{L^1(\widehat{\rho})}\|e^{-\alpha \TE }\|_{L^{1}(\rho)}} \\&=\frac{ \widehat v e^{-\alpha\TE( \widehat v )}\iint_{\RR^{d}\times \RR^{d}}(e^{-\alpha\TE( \widehat v )}-e^{-\alpha\TE( v )})d\pi(v,\widehat v)}{\|{e^{-\alpha \TE}}\|_{L^1(\rho)}\|{e^{-\alpha \TE }}\|_{L^1(\widehat \rho)}} =: I_3\,.
	\end{align*}
	Under  Assumption \ref{asum}, one can obtain the bounds
	\begin{equation}
	|I_1|\leq \frac{e^{-\alpha \underline \TE}}{\|{e^{-\alpha \TE}}\|_{L^1(\rho)}}| v - \widehat v |\leq  C_{\alpha,\TE} |v-\widehat v|\,,
	\end{equation}
	and
	\begin{equation}
	|I_2|\leq \frac{| \widehat v |\alpha e^{-\alpha \underline \TE}L| v - \widehat v |}{\|{e^{-\alpha \TE}}\|_{L^1(\rho)}}\leq \alpha  C_{\alpha,\TE}L|\widehat v||v-\widehat v|\,.
	\end{equation}
	Similar to the estimate of $|I_2|$, we estimate $I_3$
	\begin{align}
	|I_3|\leq  \alpha  C_{\alpha,\TE}^2L|\widehat v| \iint_{\RR^{d}\times \RR^{d}} |v-\widehat v|d\pi(v,\widehat v)\,.
	\end{align}
	
	Notice that $\rho,\widehat \rho\in\mc{P}_c(\RR^d)$, so for any $1\leq p<\infty$ one has
	\begin{equation}
	\max\left\{\int_{\RR^d}  |v|^p d\rho(v),\int_{\RR^d}  |\widehat v|^p d\widehat \rho(\widehat v)\right\}<\infty\,.
	\end{equation}
	Collecting the above estimates, we obtain
	\begin{align}
	| v_{\alpha,\TE}(\rho)-v_{\alpha,\TE}(\widehat \rho)|&\leq  C_{\alpha,\TE} \iint_{\RR^d\times \RR^d}|v-\widehat v|d\pi(v,\widehat v)
	+\alpha C_{\alpha,\TE}L \iint_{\RR^d\times \RR^d}|\widehat v||v-\widehat v|d\pi(v,\widehat v)\notag\\
	&\quad+\alpha C_{\alpha,\TE}^2L\int_{\RR^{d}}|\widehat v| d\widehat \rho(\widehat v) \iint_{\RR^{d}\times \RR^{d}} |v-\widehat v|d\pi(v,\widehat v)  \notag\\
	&\leq C\left( \iint_{\RR^{d}\times \RR^{d}}|v-\widehat v|^pd\pi(v,\widehat v)\right)^{\frac{1}{p}}\,,
	\end{align}
	where $C$ depends only on $ C_{\alpha,\TE}$ and $\alpha,L$. Lastly, optimizing over all couplings $\pi$ yields  \eqref{lemstaeq}.
\end{proof}

The following theorem states the well-posedness for the mean-field dynamic \eqref{selfprocess}.
\begin{theorem}
	Let $\EE$ and $\TE$ satisfy Assumptions \ref{asum1} and \ref{asum}. Then there exists a unique process $\overline V\in \mc{C}([0,T],\RR^{d})$, $T>0$, satisfying the nonlinear SDE \eqref{selfprocess}
	\begin{equation}
	d\overline V_t=- \lambda P(\overline V_t)(\OV_t-v_{\alpha,\EE} (\rho_t)) dt + \sigma |\overline V_t - v_{\alpha,\EE}(\rho_t)  | P(\overline V_t)dB_t-\frac{\sigma^2}{2}(\overline V_t-v_{\alpha,\EE} (\rho_t) )^2\Delta \gamma (\OV_t)\nabla\gamma (\OV_t)dt\,, \notag
	\end{equation}
	in strong sense for any initial data $\OV_0\in \Gamma$ distributed according to $\rho_0\in \mc{P}(\Gamma)$, where
	\begin{equation}
	v_{\alpha,\EE}(\rho_t)  = \frac{\int_{\RR^{d}}  v  e^{-\alpha \EE( v )}\,d\rho_t}{\int_{\RR^{d}}e^{-\alpha \EE( v )}\,d\rho_t},  \notag
	\end{equation}
	and $\rho_t=\rm{law}(\overline V_t)$ for all $t\in[0,T]$.
	Moreover $\overline V_t\in\Gamma$ for all $t\in[0,T]$.
\end{theorem}
\begin{proof} 
	The proof in the following is based on the Leray-Schauder fixed point theorem, see, e.g., \cite[Chapter 10]{gilbarg2015elliptic}, and it will be carried out through 5 steps.
	
	$\bullet$ \textit{Step 1:} For some given $\xi\in \mc{C}([0,T],\RR^d)$,  a distribution $\rho_0$ on $\Gamma$ and $\overline{V}_0$ with law $\rho_0$, we can uniquely solve the SDE
	\begin{equation}\label{Rself}
	d\overline V_t=-\lambda  P_1(\overline V_t)(\OV_t-\xi_t ) dt + \sigma |\overline V_t - \xi_t | P_1(\overline V_t)dB_t-\frac{\sigma^2}{2}(\overline V_t-\xi_t )^2P_2(\overline V_t)P_3(\overline V_t)dt\,
	\end{equation}
	because the coefficients are locally Lipschitz continuous and $\xi$ is independent of $\overline{V}$. In addition, following the same argument as in \eqref{dVti2=0}  we have $d\gamma(\overline V_t)=0$, so that $\overline V_t\in \Gamma$ for all time.  This introduces $\rho_t=\rm{law}(\OV_t)$ and $\rho\in \mc{C}([0,T],\mc{P}_c(\RR^d))$. Setting $\mc{T}\xi:=v_{\alpha,\TE}(\rho)\in \mc{C}([0,T],\RR^d)$ we define the map
	\begin{equation}
	\mc{T}: \mc{C}([0,T],\RR^d)\rightarrow \mc{C}([0,T],\RR^d),\quad  \xi\mapsto \mc{T}(\xi):=v_{\alpha,\TE}(\rho)\,
	\end{equation}
	which we will show to be compact in the next step.\\
	
	$\bullet$ \textit{Step 2:} For any $0<s<t\leq T$, one has
	\begin{align}
	\OV_t-\OV_s=-\lambda \int_s^tP_1(\overline V_\tau)(\OV_\tau-\xi_\tau)  d\tau+\int_s^t\sigma |\overline V_\tau - \xi_\tau | P_1(\overline V_\tau)dB_\tau-\int_s^t\frac{\sigma^2}{2}(\overline V_\tau-\xi_\tau )^2P_2(\overline V_\tau)P_3(\overline V_\tau)d\tau\,.
	\end{align}
	Then applying It\^{o}'s isometry yields
	\begin{align}
	\mathbb{E}\left[|\OV_t-\OV_s|^2\right]&\leq C_1|t-s|^2+3\sigma^2\int_s^t\mathbb{E}\left[|\overline V_\tau - \xi_\tau | ^2| P_1(\overline V_\tau)|^2\right]d\tau +C_2|t-s|^2\\
	&\leq C|t-s|\,.
	\end{align}
	where the constants $C_1,C_2,C$ only depend on $\sigma,d,T,\lambda,\|P_1\|_\infty,\|P_2\|_\infty,\|P_3\|_\infty\|\xi\|_\infty$. Therefore, by definition of the Wasserstein distance we have
	\begin{equation}
	W_2(\rho_t,\rho_s)\leq C|t-s|^{\frac{1}{2}}.
	\end{equation}
	By an application of Lemma \ref{lemsta} one obtains
	\begin{equation}
	|v_{\alpha,\TE}(\rho_t)-v_{\alpha,\TE}(\rho_s)|\leq C|t-s|^{\frac{1}{2}}\,.
	\end{equation}
	This provides the H\"{o}lder continuity of $t\to v_{\alpha,\TE}(\rho_t)$.
	Thus one has $\mc{T}(\mc{C}([0,T],\RR^d))\subset \mc{C}^{0,\frac{1}{2}}([0,T],\RR^d)\hookrightarrow \mc{C}([0,T],\RR^d)$, which implies the compactness of the map $\mc{T}$.\\
	
	$\bullet$ \textit{Step 3:} Let us define the set
	\begin{equation}
	\mathcal A:=\left\{\xi \in \mc{C}([0,T],\RR^d): \xi=\vartheta \mc{T}\xi\mbox{ for some }0\leq \vartheta\leq 1 \right\}\,.
	\end{equation}
	For $\xi\in \mathcal A$, there exists some $\OV_t$ satisfying \eqref{Rself} with its law  $\rho\in \mc{C}([0,T],\mc{P}_c(\RR^d))$ such that $\xi=\vartheta v_{\alpha,\TE}(\rho)$. Due to \eqref{esv}, we have that for any $t\in[0,T]$
	\begin{equation}
	|\xi_t|^2=\vartheta^2|v_{\alpha,\TE}(\rho_t)|^2\leq \vartheta^2\left( C_{\alpha,\TE}\int_{\RR^d}|v|d\rho\right)^2<\infty\,,
	\end{equation}
	since $\rho$ has a compact support, which verifies the boundedness of the set $\mathcal A$. Applying the Leray-Schauder fixed point theorem,  there exists a fixed point $\xi$ for the  mapping $\mc{T}$ and thereby a solution of 
	\begin{equation}\label{Rnonlinear}
	d\overline V_t= -\lambda P_1(\overline V_t)(\OV_t-v_{\alpha,\TE}(\rho_t) ) dt + \sigma |\overline V_t - v_{\alpha,\TE}(\rho_t) | P_1(\overline V_t)dB_t-\frac{\sigma^2}{2}(\overline V_t-v_{\alpha,\TE}(\rho_t) )^2P_2(\overline V_t)P_3(\overline V_t)dt
	\end{equation}
	with $\mbox{law}(\OV_t)=\rho_t$. \\
	
	$\bullet$ \textit{Step 4:}  In this step, we shall prove the uniqueness. Suppose we have two fixed points $\xi^1$ and $\xi^2$, and their corresponding process $\OV_t^1$, $\OV_t^2$ satisfying \eqref{Rself} respectively, denote by $\rho_t^1$ and $\rho_t^2$ their laws. We compute the difference $\overline Z_t:=\OV_t^1-\OV_t^2$ and applying It\^{o}'s isometry yields
	\begin{align}
	\mathbb{E}[|\overline Z_t|^2]\leq C \left ( \mathbb{E}[|\overline Z_0|^2]+\int_0^t\mathbb{E}[|\overline Z_s|^2]ds+\int_0^t|\xi^1_s-\xi^2_s|^2ds \right )\,,
	\end{align}
	where $C$ depends on $\lambda, \sigma,d,t,\|\nabla P_3\|_\infty,\| P_3\|_\infty,\|\nabla P_2\|_\infty,\| P_2\|_\infty,\|\nabla P_1\|_\infty,\| P_1\|_\infty,\|\xi^1\|_\infty,\|\xi^2\|_\infty$.
	According to Lemma \ref{lemsta}, one has
	\begin{equation}\label{xi-xi}
	|\xi^1_s-\xi^2_s|^2= |v_{\alpha,\TE}(\rho^1_s)-v_{\alpha,\TE}(\rho^2_s)|^2\leq CW_2^2(\rho_s^1,\rho_s^2)\leq C\mathbb{E}[|\overline Z_s|^2]\,,
	\end{equation}
	which implies that
	\begin{align}
	\mathbb{E}[|\overline Z_t|^2]\leq C\mathbb{E}[|\overline Z_0|^2]+C\int_0^t\mathbb{E}[|\overline Z_s|^2]ds\,.
	\end{align}
	Therefore, applying Gronwall's inequality with $\mathbb{E}[|\overline Z_0|^2]=0$, one obtain $\mathbb{E}[|\overline Z_t|^2]=0$ for all $t\in[0,T]$.
	This leads to $\xi^1\equiv \xi^2$ by \eqref{xi-xi}. Hence, we obtain the uniqueness for solutions to \eqref{Rnonlinear}.\\
	
	$\bullet$ \textit{Step 5:}  
	Similar to the argument in Theorem \ref{thmwellposednessofparticle}, the unique solution to the regularized SDE \eqref{Rnonlinear} is also the unique solution to the nonlinear SDE  \eqref{selfprocess} due to the fact that   $\gamma(\OV_t)=0$  for all $t\in[0,T].$
\end{proof}

\subsection{Well-posedness for the PDE \eqref{PDE}}

Let $\rho_0 \in \Gamma$, and let $\{\overline V_t: t\geq 0 \}$ be the solution to \eqref{selfprocess} obtained in the last section with the initial data $\OV_0$ distributed according to $\rho_0$.  For any $\varphi\in C_c^\infty(\RR^d)$, it follows from It\^{o}'s formula \eqref{Mulito} that
\begin{align}\label{pippoeq}
d \varphi(\OV_t)&=\nabla\varphi(\OV_t)\cdot \left(-\lambda P(\overline V_t)(\OV_t-v_{\alpha,\EE}(\rho_t))  -\frac{\sigma^2}{2}(\overline V_t-v_{\alpha,\EE}(\rho_t)  )^2\Delta \gamma(\OV_t)\nabla\gamma(\OV_t)\right)dt \notag\\
&\quad+\sigma |\overline V_t - v_{\alpha,\EE}(\rho_t)  | \nabla\varphi(\OV_t)\cdot P(\overline V_t)dB_t\notag\\
&\quad+\frac{\sigma^2}{2}\left(\overline V_t - v_{\alpha,\EE}(\rho_t)  \right)^2\nabla^2\varphi(\OV_t):\left (I-\nabla\gamma(\OV_t)\nabla\gamma(\OV_t)^T\right)dt\,,
\end{align}
where $\nabla^{2}\varphi$ is the Hessian and $A:B := \operatorname{Tr}(A^{T} B)$.  Taking expectation on both sides of \eqref{pippoeq}, the law $\rho_t$ of $\OV_t$ as a measure on $\RR^d$ satisfies
\begin{align}\label{wholeweak}
\frac{d}{dt}\int_{\RR^d}\varphi(v)d\rho_t(v)&=\int_{\RR^d}\nabla\varphi(v)\cdot \left(-\lambda(I-\nabla\gamma(v)\nabla\gamma(v)^T)(v-v_{\alpha,\EE}(\rho_t)) - \frac{\sigma^2}{2}(v-v_{\alpha,\EE}(\rho_t)  )^2\Delta \gamma(v)\nabla\gamma(v)\right)d\rho_t(v)
\notag\\
&\quad+\int_{\RR^d}\frac{\sigma^2}{2}\left(v- v_{\alpha,\EE}(\rho_t)  \right)^2\nabla^2\varphi(v):\left (I-\nabla\gamma(v) \nabla\gamma(v)^T\right)d\rho_t(v)\,.
\end{align}
As we have proved that $\OV_t\in \Gamma$, almost surely, that is, the density $\rho_t$ is concentrated on $\Gamma$ for any $t$, we have $\operatorname{supp}(\rho_t)\subset \Gamma$. Let us now define the restriction $\mu_t$ of $\rho_t$ on $\Gamma$ by
\begin{equation}
\int_{\Gamma}\Phi(v)d \mu_t(v)=\int_{\RR^d}\varphi(v)d\rho_t(v)
\end{equation}
for all continuous maps $\Phi\in \mc{C}(\Gamma)$, where $\varphi\in \mc{C}_b(\RR^d)$ equals $\Phi$ on $\Gamma$.

Next we define the projection
\begin{equation*}
\Pi_\Gamma (v)=v-\gamma(v)\nabla\gamma(v)\in \Gamma,\quad \mbox{for }v\in \RR^d\,.
\end{equation*}
We then let $\Gamma_\delta \subset\RR^d$ be a strip of width $\delta>0$ about $\Gamma$, where $\delta>0$ is sufficiently small to ensure that the decomposition 
\begin{equation}
v=\Pi_\Gamma (v)+\gamma(v)\nabla\gamma(v)
\end{equation} 
is unique for $v\in \Gamma_\delta$. We know such $\delta$ exists since $\gamma \in \mathcal{C}^2(\OG)$, see for example \cite[Section 2.1]{demlow2007adaptive}. 
Let now $\Phi\in \mc{C}^\infty(\Gamma)$ and define
a function $\varphi\in \mc{C}_c^\infty(\RR^d)$ such that 
\begin{equation}
\varphi(v)= \Phi\left(\Pi_\Gamma (v)\right)\quad  \mbox{ for all } v\in\Gamma_\delta\,.
\end{equation}
Then $\varphi$ defined above  is $0$-homogeneous in $v$ in the strip $\Gamma_\delta$, so that $\nabla\varphi(v)\cdot \nabla\gamma(v)=0$ for all $v$ in the support $\Gamma$ of $\rho_t$, which leads to $\nabla^2\varphi(v):\nabla\gamma(v)\nabla\gamma(v)^T=0$.  Hence,
\begin{align}
&\frac{d}{dt}\int_{\Gamma}\Phi(v)d\mu_t(v)=\frac{d}{dt}\int_{\RR^d}\varphi(v)d\rho_t(v)\notag\\
=&-\lambda \int_{\RR^{d}}\nabla\varphi(v)\cdot \left((I-\nabla\gamma(v)\nabla\gamma(v)^T)(v-v_{\alpha,\EE}(\rho_t)) \right)d\rho_t(v)+\int_{\RR^{d}}\frac{\sigma^2}{2}\left(v- v_{\alpha,\EE}(\rho_t)  \right)^2\Delta\varphi(v)d\rho_t(v)\,.\notag\,
\end{align}
\noindent
Let us now relate the Euclidean differential operators to corresponding operators on $\Gamma$, so that for $v\in\Gamma$  it holds 
$\nabla_{\Gamma} \Phi(v) = \nabla \varphi(v)$ and $\Delta_{\Gamma}\Phi(v)=\Delta \varphi(v)$. Therefore, we obtain
\begin{align}
\frac{d}{dt}\int_{\Gamma}\Phi(v)d\mu_t(v)&=-\lambda \int_{\Gamma}\nabla_{\Gamma}\Phi(v)\cdot \left((I-\nabla\gamma(v)\nabla\gamma(v)^T)(v-v_{\alpha,\EE}(\mu_t))\right)d\mu_t(v)\notag\\
&\quad +\int_{\Gamma}\frac{\sigma^2}{2}\left(v- v_{\alpha,\EE}(\mu_t) \right)^2\Delta_{\Gamma}\Phi(v)d\mu_t(v)\,,
\end{align}
where
\begin{equation}
v_{\alpha,\EE}(\mu_t) = \frac{\int_{\Gamma} v e^{-\alpha \EE(v)}\,d \mu_t}{\int_{\Gamma}e^{-\alpha \EE(v)}\,d \mu_t}.
\end{equation}
Thus by this construction we obtain a weak solution $\mu_t$ to the PDE \eqref{PDE}.\\

Next we prove the uniqueness of solutions to \eqref{PDE}. Assume that $\rho_t^1$ and $\rho_t^2$ are two solutions to \eqref{PDE} with the same initial data $\rho_0$, and at each time $t$ we treat them as measures on $\RR^d$ concentrated on the hypersurface $\Gamma$. Then we construct two linear process $(\OV_t^i)_{t\geq 0}$ $(i=1,2)$ satisfying
\begin{equation}
d\overline V_t^i=-\lambda P_1(\overline V_t^i)(\OV_t^i-v_{\alpha,\EE} (\rho_t^i)) dt + \sigma |\overline V_t^i - v_{\alpha,\EE}(\rho_t^i)  | P_1(\overline V_t^i)dB_t-\frac{\sigma^2}{2}(\overline V_t^i-v_{\alpha,\EE} (\rho_t^i) )^2P_2(\OV_t^i)P_3(\OV_t^i)dt\,, \label{Eqn for Vbar}
\end{equation}
with the common initial data $\OV_0$ distributed  according to $\rho_0$. Let us denote $\rm{law}(\OV_t^i)=\bar{\rho}_t^i$ $(i=1,2)$ as measures  on $\RR^{d}$, which are solutions to the following linear PDE
\begin{align*}
\partial_t \bar \rho_t^i&= \nabla\cdot \left(\bar \rho_t^i\left(\lambda P_1(v)(v-v_{\alpha,\EE} (\rho_t^i))+\frac{\sigma^2}{2}(v-v_{\alpha,\EE} (\rho_t^i) )^2P_2(v)P_3(v)\right)\right)\\
&\quad +\frac{\sigma^2}{2}\sum_{k,\ell=1}^{d}\frac{\partial^2}{\partial v_k\partial v_\ell}\left(|v-v_{\alpha,\EE}(\rho_t^i) |^2(P_1P_1^T)_{k\ell}\bar \rho_t^i\right)\,.
\end{align*}
Since the uniqueness for the above linear PDE holds and $\rho_t^i$ is also a solution to the above PDE on $\RR^d$ (see \eqref{wholeweak}), it follows that
$\bar \rho_t^i=\rho_t^i$ $(i=1,2)$. Consequently, the process $(\OV_t^i)_{(t\geq 0)}$ are solutions to the nonlinear SDE \eqref{selfprocess}, for which the uniqueness has been obtained. Hence $(\OV_t^1)_{(t\geq 0)}$ and $(\OV_t^2)_{(t\geq 0)}$ are equal, which implies $\rho_t^1=\bar \rho_t^1=\bar \rho_t^2=\rho_t^2$. Thus the uniqueness is obtained.

\section{Mean-field limit}\label{sec:MFlimit}
The well-posedness of  \eqref{stochastic Kuramoto-Vicsek}, \eqref{PDE} and \eqref{selfprocess} obtained in the last section provides all the ingredients we need for the mean-field limit. Let $((\OV_t^i)_{t\geq 0})_{i \in [N]}$ be $N$ independent copies of solutions to \eqref{selfprocess}. They are i.i.d. with the same distribution $\rho_t$. Assume that $((V_t^i)_{t\geq 0})_{i \in [N]}$ is the solution to the particle system \eqref{stochastic Kuramoto-Vicsek}. Since  $\OV_t^i,V_t^i\in \Gamma$ for all $i$ and $t$, $((\OV_t^i)_{t\geq 0})_{i \in [N]}$ and $((V_t^i)_{t\geq 0})_{i \in [N]}$ are solutions to the corresponding regularized systems \eqref{Rnonlinear} and \eqref{RSKV} respectively. We denote below $\overline \rho_t^N = \frac{1}{N} \sum_{j=1}^N \delta_{\OV_t^j}$ and $\rho_t = \rm{law}(\OV_t)$.

Before stating our theorem on the mean-field limit, let us introduce the following lemma on a large deviation bound.
\begin{lemma}\label{lemLLN}  Let $\EE$ and $\TE$ satisfy Assumptions \ref{asum1} and \ref{asum}.  Let $((\OV_t^i)_{t\geq 0})_{i \in [N]}$ be the solution to the mean-field dynamics \eqref{Rnonlinear}, which are i.i.d. with common distribution $\rho\in \mc C([0,T],\mc{P}_c(\RR^d))$. Then there exists a constant $C$ depending  only on  $\Gamma$ and $C_{\alpha,\TE}$ such that
	\begin{equation} \label{estimazza}
	\sup_{t\in[0,T]}\mathbb{E}\left[|v_{\alpha,\TE}(\overline \rho_t^N)-v_{\alpha,\TE}(\rho_t)|^2\right]\leq C N^{-1}\,.
	\end{equation}
\end{lemma}
\begin{proof}
	We start by estimating
	\begin{align}\label{splites}
	&|v_{\alpha,\TE}(\overline \rho_t^N)-v_{\alpha,\TE}(\rho_t)|= \left|\frac{\sum_{j=1}^{N} \OV_t^j e^{-\alpha \TE( \OV_t^j )}}{\sum_{i=1}^{N}e^{-\alpha \TE( \OV_t^i)  }}-\frac{1}{\int_{\RR^d}e^{-\alpha\TE(v)}d\rho_t}\int_{\RR^d} v e^{-\alpha\TE(v)}d\rho_t\right| \notag\\
	\leq & \left|\frac{1}{\frac{1}{N}\sum_{i=1}^{N} e^{-\alpha \TE( \OV_t^i )}}\left(\frac{1}{N}\sum_{j=1}^{N} \OV_t^j e^{-\alpha \TE( \OV_t^j )}-\int_{\RR^d} v e^{-\alpha\TE(v)}d\rho_t\right)\right| \notag\\
	&+\left|\int_{\RR^d} v e^{-\alpha\TE(v)}d\rho_t\left(\frac{1}{\frac{1}{N}\sum_{i=1}^{N} e^{-\alpha \TE( \OV_t^i )}}-\frac{1}{\int_{\RR^d} e^{-\alpha\TE(v)}d\rho_t}\right)\right| \notag\\
	\leq& e^{ \alpha\overline\TE}\left|\frac{1}{N}\sum_{j=1}^{N} \OV_t^j e^{-\alpha \TE( \OV_t^j )}-\int_{\RR^d} v e^{-\alpha\TE(v)}d\rho_t\right|+e^{\alpha(2\overline\TE-\underline\TE)}\left|\frac{1}{N}\sum_{j=1}^{N} e^{-\alpha \TE( \OV_t^j )}-\int_{\RR^d} e^{-\alpha\TE(v)}d\rho_t\right|\,.
	\end{align}
	Here we used explicitly the upper and lower bounds $\overline\TE, \underline\TE$.
	Denote
	\begin{equation}
	\overline Z_t^j:=\OV_t^j e^{-\alpha \TE( \OV_t^j )}-\int_{\RR^d} v e^{-\alpha\TE(v)}d\rho_t\,,
	\end{equation}
	then one can rewrite
	\begin{equation}
	\frac{1}{N}\sum_{j=1}^{N} \OV_t^j e^{-\alpha \TE( \OV_t^j )}-\int_{\RR^d} v e^{-\alpha\TE(v)}d\rho_t=\frac{1}{N}\sum_{j=1}^{N}\overline Z_t^j\,.
	\end{equation}
	It is obvious that $\BE[\overline Z_t^j]=0$, since $\{(\OV_t^j)_{t\geq 0}\}_{j=1}^N$ are i.i.d. with common distribution $\rho$. Next we compute
	\begin{align*}
	&\mathbb{E}\left[\left|\frac{1}{N}\sum_{j=1}^{N} \OV_t^j e^{-\alpha \TE( \OV_t^j )}-\int_{\RR^d} v e^{-\alpha\TE(v)}d\rho_t\right|^2\right]=\frac{1}{N^2}\BE\left[(\sum_{j=1}^N\overline Z_t^j)(\sum_{k=1}^N\overline Z_t^k)\right]\\
	=&\frac{1}{N^2}\BE\left[\sum_{j=1}^N(\overline Z_t^j)^2\right]=\frac{1}{N}\BE[(\overline Z_t^1)^2]\,,
	\end{align*}
	where we have used the fact that $\BE[Z_t^jZ_t^k]=0$ for $j\neq k$.  It is also easy to check that
	\begin{equation}
	\BE[(\overline Z_t^1)^2]\leq 2\BE\left[|\OV_t^1 e^{-\alpha \TE( \OV_t^1 )}|^2\right]+2\left(\int_{\RR^d} v e^{-\alpha\TE(v)}d\rho_t\right)^2\leq 4C_{\Gamma}e^{-2\alpha\underline\TE}\,.
	\end{equation}
	Here we have used the fact that $\OV_t^1\in \Gamma$, so it holds $\BE\left[|\OV_t^1|^2\right]=\int_{\RR^d}|v|^2d\rho_t\leq C_{\Gamma}$ for some $C_{\Gamma}>0$ depending on $\Gamma$.
	Thus we conclude
	\begin{equation}
	\mathbb{E}\left[\left|\frac{1}{N}\sum_{j=1}^{N} \OV_t^j e^{-\alpha \TE( \OV_t^j )}-\int_{\RR^d} v e^{-\alpha\TE(v)}d\rho_t\right|^2\right]\leq \frac{1}{N}4C_{\Gamma}e^{-2\alpha\underline\TE}\,.
	\end{equation}
	By following the same argument it also holds that
	\begin{equation}
	\mathbb{E}\left[\left|\frac{1}{N}\sum_{j=1}^{N} e^{-\alpha \TE( \OV_t^j )}-\int_{\RR^d}e^{-\alpha\TE(v)}d\rho_t\right|^2\right]\leq \frac{1}{N}4C_{\Gamma}e^{-2\alpha\underline\TE}\,.
	\end{equation}
	Combining the above two estimates and \eqref{splites}, one has
	\begin{align}
	\mathbb{E}\left[|v_{\alpha,\TE}(\overline \rho_t^N)-v_{\alpha,\TE}(\rho_t)|^2\right]\leq\frac{1}{2}C_{\Gamma}e^{2\alpha(\overline{\TE}-\underline{\TE})}\frac{1}{N}+\frac{1}{2}C_{\Gamma}e^{\alpha(2\overline{\TE}-3\underline \TE)}\frac{1}{N}\leq C_{\Gamma}e^{3\alpha(\overline{\TE}-\underline \TE))}\frac{1}{N}=C_{\Gamma}C_{\alpha,\TE}^3\frac{1}{N}\,.
	\end{align}
	Thus we have completed the proof.
\end{proof}

Then we get the following mean-field limit result by the classical Sznitman's theory.
\begin{theorem}\label{thmmean}
	Under the  Assumptions \ref{asum1} and \ref{asum}, for any $T>0$, let $((V_t^i)_{t\in [0,T]})_{i \in [N]}$  and $((\OV_t^i)_{t\in [0,T]})_{i \in [N]}$  be respective solutions to  \eqref{stochastic Kuramoto-Vicsek} and \eqref{selfprocess} up to time $T$ with the same initial data  $V_0^i=\OV_0^i$  and the same Brownian motions $((B_t^i)_{t\in [0,T]})_{i \in [N]}$. Then there exists a constant $C>0$ depending only on $\Gamma$, $\lambda$, $\alpha$,  $\sigma$, $\|\nabla P_1\|_\infty$, $\|P_1\|_\infty$, $\|\nabla P_2\|_\infty$, $\|P_2\|_\infty$, $\|\nabla P_3\|_\infty$, $\|P_3\|_\infty$, $L$ and $C_{\alpha,\TE}$, such that
	\begin{equation}\label{MFlimit}
	\sup_{i=1,\cdots,N}\mathbb{E}[|V_t^i-\OV_t^i|^2]\leq CT\left(1+CTe^{CT}\right)\frac{1}{N}\,,
	\end{equation}
	holds for all $0\leq t\leq T$.
\end{theorem}
\begin{remark}
	The estimate above guarantees the weak convergence of the empirical measure $\rho_t^N$ towards $\rho_t$, in the following sense
	\begin{equation}
	\sup_{t\in[0,T]}\BE\left[|\la\rho_t^N,\phi\ra-\la \rho_t,\phi\ra|^2\right]\to 0\mbox{ as }N\to \infty\quad \mbox{ for any test funtion }\phi\in \mc{C}_b^1(\RR^{d})\,.
	\end{equation} 
	Indeed, one has
	\begin{align*}
	&\BE\left[|\la\rho_t^N,\phi\ra-\la \rho_t,\phi\ra|^2\right]=\BE\left[\left|\frac{1}{N}\sum_{i=1}^{N}\phi(V_t^i)-\int_{\RR^{d}}\phi(v)d\rho_t(v)\right|^2\right]\\
	\leq& 2\BE\left[|\phi(V_t^1)-\phi(\OV_t^1)|^2\right]+2\BE\left[\left|\frac{1}{N}\sum_{i=1}^{N}\phi(\OV_t^i)-\int_{\RR^{d}}\phi(v)d\rho_t(v)\right|^2\right]\leq \frac{C}{N}\|\phi\|_{\mc{C}^1}^2
	\,.
	\end{align*} We refer to Refs. \cite{sznitman1991topics,bolley2011stochastic} for more details.
\end{remark}
\begin{proof}
	Notice that $((\OV_t^i)_{t\geq 0})_{i \in [N]}$  and $((V_t^i)_{t\geq 0})_{i \in [N]}$  are also solutions to the corresponding regularized systems \eqref{Rnonlinear} and \eqref{RSKV} respectively, and
	\begin{align}
	d(V_t^i-\OV_t^i)&=-\lambda\left(  P_1(V_t^i)(V_t^i-v_{\alpha,\TE}(\rho_t^N))-P_1(\overline V_t^i)(\OV_t^i-v_{\alpha,\TE}(\rho_t)) \right)dt \notag\\
	&+\sigma \left(|V_t^i - v_{\alpha,\TE}(\rho_t^N)| P_1(V_t^i)- |\overline V_t^i- v_{\alpha,\TE}(\rho_t) | P_1(\overline V_t^i)\right)dB_t^i\notag\\
	&-\frac{\sigma^2}{2}\left((V_t^i-v_{\alpha,\TE}(\rho_t^N))^2P_2(V_t^i)P_3(V_t^i)-(\overline V_t^i-v_{\alpha,\TE}(\rho_t) )^2P_2(\overline V_t^i)P_3(\overline V_t^i)\right)dt
	\,.
	\end{align}
	By applying It\^{o}'s formula one has
	\begin{align}\label{expec}
	d(V_t^i-\OV_t^i)^2&=-2 \lambda(V_t^i-\OV_t^i) \cdot\left(  P_1(V_t^i)(V_t^i-v_{\alpha,\TE}(\rho_t^N))-P_1(\overline V_t^i)(\OV_t^i-v_{\alpha,\TE}(\rho_t)) \right)dt \notag\\
	&+2 \sigma  (V_t^i-\OV_t^i) \cdot \left(|V_t^i - v_{\alpha,\TE}(\rho_t^N)| P_1(V_t^i)- |\overline V_t^i- v_{\alpha,\TE}(\rho_t) | P_1(\overline V_t^i)\right)dB_t^i\notag\\
	&-\sigma^2 (V_t^i-\OV_t^i)\cdot\left((V_t^i-v_{\alpha,\TE}(\rho_t^N))^2P_2(V_t^i)P_3(V_t^i)-(\overline V_t^i-v_{\alpha,\TE}(\rho_t) )^2P_2(\overline V_t^i)P_3(\overline V_t^i)\right)dt \notag\\
	&+\sigma^2\sum_{k=1}^d\left(|V_t^i - v_{\alpha,\TE}(\rho_t^N)| P_1^k(V_t^i)- |\overline V_t^i- v_{\alpha,\TE}(\rho_t) | P_1^k(\overline V_t^i)\right)\notag\\
	&\cdot \left(|V_t^i - v_{\alpha,\TE}(\rho_t^N)| P_1^k(V_t^i)- |\overline V_t^i- v_{\alpha,\TE}(\rho_t) | P_1^k(\overline V_t^i)\right)dt
	\,,
	\end{align}
	where $P_1^k(\cdot)$ represents the $k$-th row of the matrix $P_1(\cdot)$. Using Lemma \ref{lemlocal} it is easy to compute that
	\begin{align*}
	&\left|P_1(V_t^i)(V_t^i-v_{\alpha,\TE}(\rho_t^N))-P_1(\overline V_t^i)(\OV_t^i-v_{\alpha,\TE}(\rho_t)) \right|\\
	\leq&\left | P_1(V_t^i)v_{\alpha,\TE}(\rho_t^N)-P_1(\overline V_t^i)v_{\alpha,\TE}(\rho_t) \right |+(\|P_1\|_\infty+\|\nabla P_1\|_\infty)|V_t^i-\OV_t^i|\\
	\leq &\left |P_1(V_t^i)\left(v_{\alpha,\TE}(\rho_t^N)-v_{\alpha,\TE}(\overline \rho_t^N)\right)+v_{\alpha,\TE}(\overline \rho_t^N)\left(P_1(V_t^i)-P_1(\overline V_t^i)\right)+P_1(\overline V_t^i)\left(v_{\alpha,\TE}(\overline \rho_t^N)-v_{\alpha,\TE}(\rho_t)\right) \right |\\
	&+(\|P_1\|_\infty+\|\nabla P_1\|_\infty)|V_t^i-\OV_t^i|\\
	\leq&\|P_1\|_\infty \left(\frac{C_{\alpha,\TE}}{N}+\frac{2\alpha LC_{\alpha,\TE}}{N}\|\mb{\OV}_t^N\|_\infty\right)\|\mb{V}_t^N- \mb{\OV}_t^N\|_1+\|\nabla P_1\|_\infty\frac{C_{\alpha,\TE}\|\mb{\OV}_t^N\|_1}{N}|V_t^i-\OV_t^i|\\
	&+\|P_1\|_\infty |v_{\alpha,\TE}(\overline \rho_t^N)-v_{\alpha,\TE}(\rho_t)|+(\|P_1\|_\infty+\|\nabla P_1\|_\infty)|V_t^i-\OV_t^i| \\
	\leq & C(\alpha,\|P_1\|_\infty,L,C_{\alpha,\TE},C_{\Gamma})\frac{\sum_{i=1}^N|V_t^i-\OV_t^i|}{N}+C(C_{\alpha,\TE},C_{\Gamma},\|\nabla P_1\|_\infty,\|P_1\|_\infty)|V_t^i-\OV_t^i|+\|P_1\|_\infty |v_{\alpha,\TE}(\overline \rho_t^N)-v_{\alpha,\TE}(\rho_t)|\,,
	\end{align*}
	where we have used the fact that  $V_t^i,\OV_t^i\in\Gamma$ with $\Gamma$ being compact, so there exists some $C_\Gamma>0$ such that
	\begin{equation*}
	|V_t^i|,|\OV_t^i|\leq C_{\Gamma},\quad \frac{\|\mb{\OV}_t^N\|_1}{N}=\frac{\sum_{i=1}^N|\OV_t^i|}{N}\leq C_{\Gamma},\quad \|\mb{\OV}_t^N\|_\infty=\sup_{i=1,\cdots,N}|\OV_t^i|\leq C_{\Gamma}\,.
	\end{equation*}
	Hence it yields
	\begin{align}\label{ter1}
	&-2\lambda(V_t^i-\OV_t^i) \cdot \left(  P_1(V_t^i)(V_t^i-v_{\alpha,\TE}(\rho_t^N))-P_1(\overline V_t^i)(\OV_t^i-v_{\alpha,\TE}(\rho_t)) \right)\notag\\
	\leq & C \left ( \frac{\sum_{i=1}^N|V_t^i-\OV_t^i|^2}{N}+ |V_t^i-\OV_t^i|^2+  |v_{\alpha,\TE}(\overline \rho_t^N)-v_{\alpha,\TE}(\rho_t)|^2 \right)\,,
	\end{align}
	where we have used Cauchy's inequality and $C$ depends only on $\alpha,\|\nabla P_1\|_\infty, \|P_1\|_\infty,L,C_{\Gamma}$ and $C_{\alpha,\TE}$, and $\|P_1\|_\infty$ is the spectral norm of the matrix $P_1$.
	
	Similarly we compute
	{\small 	\begin{align*}
		&\left |(V_t^i-v_{\alpha,\TE}(\rho_t^N))^2P_2(V_t^i)P_3(V_t^i)-(\overline V_t^i-v_{\alpha,\TE}(\rho_t) )^2P_2(\overline V_t^i)P_3(\overline V_t^i)\right |\\
		=&\left | (V_t^i-v_{\alpha,\TE}(\rho_t^N))^2\left(P_2(V_t^i)P_3(V_t^i)-P_2(\OV_t^i)P_3(\OV_t^i)\right)
		+P_2(\OV_t^i)P_3(\OV_t^i)\left((V_t^i-v_{\alpha,\TE}(\rho_t^N))^2-(V_t^i-v_{\alpha, \TE}(\overline\rho_t^N))^2\right)\right . \\
		&+ \left . P_2(\OV_t^i)P_3(\OV_t^i)\left((V_t^i-v_{\alpha, \TE}(\overline\rho_t^N))^2-(\overline V_t^i-v_{\alpha,\TE}(\rho_t) )^2\right) \right |\\
		\leq &C|V_t^i-\OV_t^i|+C\frac{\sum_{i=1}^N|V_t^i-\OV_t^i|}{N}+C|v_{\alpha,\TE}(\overline \rho_t^N)-v_{\alpha,\TE}(\rho_t)|\,,
		\end{align*} }
	where we have used Lemma \ref{lemlocal} and the fact that $|V_t^i|,|\OV_t^i|\leq C_\Gamma$.
	Here $C$ depends only on $\alpha,\|\nabla P_2\|_\infty, \|P_2\|_\infty$, $\|\nabla P_3\|_\infty, \|P_3\|_\infty,L,C_{\Gamma}$ and $C_{\alpha,\TE}$.  This leads to
	\begin{align}\label{ter2}
	&-\sigma^2 (V_t^i-\OV_t^i)\cdot\left((V_t^i-v_{\alpha,\TE}(\rho_t^N))^2P_2(V_t^i)P_3(V_t^i)-(\overline V_t^i-v_{\alpha,\TE}(\rho_t) )^2P_2(\overline V_t^i)P_3(\overline V_t^i)\right)\notag\\
	\leq &C \left (\frac{\sum_{i=1}^N|V_t^i-\OV_t^i|^2}{N}+|V_t^i-\OV_t^i|^2+ |v_{\alpha,\TE}(\overline \rho_t^N)-v_{\alpha,\TE}(\rho_t)|^2 \right)\,,
	\end{align}
	where $C$ depends only on $\alpha, \lambda,\sigma,\|\nabla P_2\|_\infty, \|P_2\|_\infty,\|\nabla P_3\|_\infty, \|P_3\|_\infty,L,C_{\Gamma}$ and $C_{\alpha,\TE}$. 
	
	Now let us compute
	\begin{align*}
	&\left | |V_t^i - v_{\alpha,\TE}(\rho_t^N)| P_1^k(V_t^i)- |\overline V_t^i- v_{\alpha,\TE}(\rho_t) | P_1^k(\overline V_t^i) \right |\\
	=&\left | |V_t^i - v_{\alpha,\TE}(\rho_t^N)|\left(P_1^k(V_t^i)-P_1^k(\OV_t^i)\right)+P_1^k(\OV_t^i)\left(|V_t^i - v_{\alpha,\TE}(\rho_t^N)|-|V_t^i - v_{\alpha,\TE}(\overline \rho_t^N)|\right) \right .\\
	&\left . + P_1^k(\OV_t^i)\left(|V_t^i - v_{\alpha,\TE}(\overline \rho_t^N)|-|\overline V_t^i- v_{\alpha,\TE}(\rho_t) |\right)\right |\\
	\leq&C \left (|V_t^i-\OV_t^i|+\frac{\sum_{i=1}^N|V_t^i-\OV_t^i|}{N}+|v_{\alpha,\TE}(\overline \rho_t^N)-v_{\alpha,\TE}(\rho_t)| \right)\,,
	\end{align*}
	where $C$ depends only on $\alpha,\|\nabla P_1\|_\infty, \|P_1\|_\infty,L,C_{\Gamma}$ and $C_{\alpha,\TE}$. This implies that
	{\small \begin{align}\label{ter3}
		&\sigma^2\sum_{k=1}^d\left(|V_t^i - v_{\alpha,\TE}(\rho_t^N)| P_1^k(V_t^i)- |\overline V_t^i- v_{\alpha,\TE}(\rho_t) | P_1^k(\overline V_t^i)\right)\cdot \left(|V_t^i - v_{\alpha,\TE}(\rho_t^N)| P_1^k(V_t^i)- |\overline V_t^i- v_{\alpha,\TE}(\rho_t) | P_1^k(\overline V_t^i)\right)\notag\\
		\leq &C \left (\frac{\sum_{i=1}^N|V_t^i-\OV_t^i|^2}{N}+|V_t^i-\OV_t^i|^2+|v_{\alpha,\TE}(\overline \rho_t^N)-v_{\alpha,\TE}(\rho_t)|^2 \right)\,,
		\end{align} }	
	where $C$ depends only on $\alpha,\sigma,\|\nabla P_1\|_\infty, \|P_1\|_\infty,L,C_{\Gamma}$ and $C_{\alpha,\TE}$. 
	
	Taking expectation of both sides of \eqref{expec}, and collecting estimates \eqref{ter1}, \eqref{ter2} and \eqref{ter3}, one concludes
	\begin{align*}
	\mathbb{E}[|V_t^i-\OV_t^i|^2]&\leq\mathbb{E}[|V_0^i-\OV_0^i|^2]+C\int_0^t\frac{\sum_{i=1}^N\mathbb{E}[|V_s^i-\OV_s^i|^2]}{N}ds+C\int_0^t\mathbb{E}[|V_s^i-\OV_s^i|^2]ds\\
	&+C\int_0^t\mathbb{E}\left[|v_{\alpha,\TE}(\overline \rho_s^N)-v_{\alpha,\TE}(\rho_s)|^2\right] ds\\
	&\leq \mathbb{E}[|V_0^i-\OV_0^i|^2]+C\int_0^t\sup_{i=1,\cdots,N}\mathbb{E}[|V_s^i-\OV_s^i|^2]ds+CT\frac{1}{N}\,,
	\end{align*}
	where we have used  the fact that
	\begin{align*}
	\BE\left[(V_t^i-\OV_t^i) \cdot \left(|V_t^i - v_{\alpha,\TE}(\rho_t^N)| P_1(V_t^i)- |\overline V_t^i- v_{\alpha,\TE}(\rho_t) | P_1(\overline V_t^i)\right)dB_t^i\right]=0\,,
	\end{align*}
	and Lemma \ref{lemLLN} in the last inequality.
	Applying Gronwall's inequality with $\mathbb{E}[|V_0^i-\OV_0^i|^2]=0$, one concludes
	\begin{equation}
	\sup_{i=1,\cdots,N}\mathbb{E}[|V_t^i-\OV_t^i|^2]\leq C T\left(1+CTe^{CT}\right)\frac{1}{N}\,,
	\end{equation}
	for all $t\in[0,T]$, which completes the proof.
\end{proof}

\begin{remark}
	The constant $C>0$ appearing in the estimate \eqref{MFlimit} may depend on the dimension through the norm of $P_2$ or $\nabla P_2$, see \eqref{ter2}. Nevertheless we expect this dependency to scale
	at most linearly as $d-1$. In fact, for the case of  sphere $\Gamma=\mathbb S^{d-1}$ it is $P_2(v)=\Delta\gamma(v)= \frac{d-1}{|v|}$. We conclude that in general there is no curse of dimensionality involved in  the estimate \eqref{MFlimit}.
\end{remark}

\section{Conclusions}

We presented a new consensus-based model for global optimization on hypersurfaces, which is inspired by the kinetic Kolmogorov-Kuramoto-Vicsek equation on the sphere. The main results of this paper are the well-posedness of the stochastic particle system and its mean-field limit, which is obtained by the coupling method through introducing an auxiliary self-consistent nonlinear SDE. The presented mean-field limit is not affected by curse of dimension, i.e., the rate of convergence is of order $N^{-1}$ in the particle number $N$. This favorable theoretical rate is confirmed by numerical experiments in very high dimension ($d \approx 3000$) \cite{fhps20-2}.  In the related paper Ref. \cite{fhps20-2}  we analyze the large time behavior of the system for  $\Gamma = \mathbb S^{d-1}$, we prove its convergence to global minimizers in a suitable sense, and we provide several applications in machine learning.  We consider the results of this paper as a preliminary step towards the formulation of consensus-based optimization on compact manifolds \cite{emery1989stochastic}, which requires  a more general approach for the analysis of the global large time asymptotics.

\section{Appendix}
\begin{theorem}[Multidimensional It\^{o}'s formula]
	Let
	$$dX_t=\textbf{u}(t)dt+\textbf{v}(t)dB_t$$
	be an $d$-dimensional It\^{o} process, where $X_t\in \RR^{d}$, $\textbf{u}(t)\in\RR^d$, $\textbf{v}(t)\in\RR^{d\times d'}$, and $B_t$ is $d'$-dimensional Brownian motion. Assume $\varphi(x)$ be a $\mc{C}^2$ map from $\RR^d$ to $\RR$, then it holds that
	\begin{align}\label{Mulito}
	\varphi(X_t)&=\varphi(X_0)+\int_0^t\left(\nabla \varphi(X_s)\cdot \textbf{u}(s) +\frac{1}{2}\sum_{i,j=1}^{d}\frac{\partial^2\varphi}{\partial x_i\partial x_j}(X_s)v_i(s)v_j(s)^T\right)ds\notag\\
	&\quad+\int_0^t\nabla\varphi(X_s)\cdot \textbf{v}(s)dB_s
	\end{align}
	with $v_i(s)$ being the $i$-th row of the matrix $\textbf{v}(s)$.
\end{theorem}

\section*{Acknowledgment}

Massimo Fornasier and Hui Huang acknowledge the support of the DFG Project "Identification of Energies from Observation of Evolutions" and the DFG SPP 1962 "Non-smooth and Complementarity-based Distributed Parameter Systems: Simulation and Hierarchical Optimization".     The present project and Philippe  S\"{u}nnen  are supported by the National Research Fund, Luxembourg (AFR PhD Project Idea ``Mathematical Analysis of Training Neural Networks'' 12434809). 
Lorenzo Pareschi acknowledges the support of the John Von Neumann guest Professorship program of the Technical University of Munich and the PRIN Project 2017KKJP4X "Innovative numerical methods for evolutionary partial differential equations and applications".

\bibliographystyle{plain}
\bibliography{bibfile}
%%%%%%%%%%%%%%%%%%%%%%%%
%\begin{thebibliography}{99}
%\end{thebibliography}
%%%%%%%%%%%%%%%%%%%%%%%%

\end{document}